\documentclass[12pt]{amsart} 
\usepackage[left=3cm,top=2.5cm,bottom=2.5cm,right=3cm]{geometry}
\usepackage{times}
\usepackage{amssymb, amsmath, amsthm}
\usepackage{mathtools}
\usepackage{graphicx,xspace}
\usepackage{epsfig}
\usepackage{enumitem}
\usepackage[usenames,dvipsnames]{xcolor}
\usepackage{tikz}
\usepackage{gensymb}
\usepackage[T1]{fontenc}
\usepackage[utf8]{inputenc}
\usepackage{bm}
\usepackage{subcaption}

\usepackage{mathrsfs}
\definecolor{green}{RGB}{0,127,0}
\definecolor{red}{RGB}{191,0,0}
\usepackage[colorlinks,cite color=red,link color=green,pagebackref=true]{hyperref}
\usepackage{todonotes}

\usepackage{amsmath}
\usepackage{enumitem}
\usepackage{amssymb,amsfonts,stmaryrd,mathrsfs}
\usepackage{amsmath,amsthm,mathtools,bbm}
\usepackage{pgf,tikz} 
\usepackage{subcaption}
\usepackage{hyperref}
\usepackage{todonotes}
\usepackage[capitalize]{cleveref}
\newtheorem{thm}{Theorem}[section]
\newtheorem{cor}[thm]{Corollary}
\newtheorem{lem}[thm]{Lemma}
\newtheorem{prop}[thm]{Proposition}
\newtheorem{defi}[thm]{Definition}
\newtheorem{conj}{Conjecture}
\newtheorem{defprop}[thm]{Definition-Proposition}

\theoremstyle{remark}
\newtheorem{rmq}{Remark}

\newcommand{\bfp}{\mathbf{p}}
\newcommand{\bfq}{\mathbf{q}}
\newcommand{\bfr}{\mathbf{r}}
\newcommand{\bfs}{\mathbf{s}}
\newcommand{\B}{\mathcal B}
\newcommand{\Binf}{\mathcal B_\infty}
\newcommand{\Binfs}{\mathcal B_\infty^{>}}
\newcommand{\C}{\mathcal C}
\newcommand{\tY}{\widetilde{Y}}
\newcommand{\tZ}{\widetilde{Z}}
\newcommand{\ad}{\mathrm{ad}}
\newcommand{\Span}{\mathrm{Span}}
\newcommand{\mcP}{\mathcal{P}}
\newcommand{\PY}{\mathcal{P}_{Y}}
\newcommand{\PZ}{\mathcal{P}_{Z}}
\newcommand{\PYZ}{\mathcal{P}_{\tY,\tZ}}
\newcommand{\Jxi}{J^{(\alpha)}_{\xi}}
\newcommand{\Jxish}{J^{(\alpha)\#}_{\xi}}
\newcommand{\Ps}{\mathcal{P}_{\leq s}}
\newcommand{\tF}{\widetilde F}
\newcommand{\Jch}{\theta^{(\alpha)}}
\newcommand{\Jla}{J^{(\alpha)}_{\lambda}}
\newcommand{\tJch}{\tilde{\theta}^{(\alpha)}}

\newcommand{\Vcirc}{\mathcal{V}_\circ}
\newcommand{\Vbul}{\mathcal{V}_\bullet}
\newcommand{\cc}{cc}
\newcommand{\tf}{\nu_\diamond}
\newcommand{\Mk}{\mathcal{M}^{(k)}}
\newcommand{\Minf}{\mathcal{M}^{(\infty)}}
\newcommand{\GY}{\Gamma_Y}
\newcommand{\GtY}{\Gamma_{\widetilde Y}}
\newcommand{\GtZ}{\Gamma_{\widetilde Z}}

\DeclareMathOperator{\lex}{lex}
\DeclareMathOperator{\stat}{stat}

\newcommand{\Shifted}{\mathcal{S}^*_\alpha}
\newcommand{\Poly}{\mathscr{P}^*_\gamma}
\newcommand{\Sym}{\mathcal{S}_\alpha}
\DeclareMathOperator{\Ch}{Ch}
\DeclareMathOperator{\RS}{RS}
\DeclareMathOperator{\CS}{CS}
\DeclareMathOperator{\ct}{ct}
\newcommand{\Cha}{\Ch^{(\alpha)}}
\newcommand{\QQ}{\mathbb{Q}}
\newcommand{\YY}{\mathbb{Y}}
\newcommand{\RR}{\mathbb{R}}
\newcommand{\ZZ}{\mathbb{Z}}
\newcommand{\NN}{\mathbb{N}}
\newcommand{\pp}{\mathbf{p}}
\newcommand{\SSS}{\mathbf{S}}

\newcommand{\ribbon}{\vec{\bm{\Gamma}}}
\newcommand{\la}{\lambda}
\def\a{\alpha}

\author[H.~Ben Dali]{Houcine Ben Dali}
\address{Université de Lorraine, CNRS, IECL, F-54000 Nancy;
  Universit\'e de Paris, CNRS, IRIF, F-75006 Paris, France.
}
\email{houcine.ben-dali@univ-lorraine.fr}
 
\author[M.~Dołęga]{Maciej Dołęga}
\address{
Institute of Mathematics, 
Polish Academy of Sciences, 
ul. Śniadeckich 8, 
00-956 Warszawa, Poland.
}
\email{mdolega@impan.pl}

\thanks{MD is supported by {\it Narodowe Centrum Nauki},
  grant 2021/42/E/ST1/00162.}

\title[Positive formula
for Jack polynomials and Jack characters]{Positive formula
for Jack polynomials, Jack characters and proof of Lassalle's conjecture}

\begin{document}
\begin{abstract}
We give an explicit formula for the power-sum expansion of Jack
polynomials. We deduce it from a more general formula, which we provide
here, that interprets Jack characters in terms of bipartite maps. We
prove Lassalle's conjecture from 2008 on integrality and positivity of
Jack characters in Stanley's coordinates and give a new formula for
Jack polynomials using creation operators.
\end{abstract}

\maketitle

\section{Introduction}
\subsection{Jack polynomials and a one-parameter deformation of Young's formula}\label{ssec Jack}

Jack polynomials $J_\lambda^{(\alpha)}$ are symmetric functions
indexed by an integer partition $\lambda$ and a deformation parameter
$\alpha$. They interpolate, up to scaling factors, between Schur
functions for $\alpha=1$ and zonal polynomials for
$\alpha=2$. Originally, they were introduced by Jack
in~\cite{Jack1970/1971} as an important tool in statistics, but it
turned out that they appear quite naturally in many different
contexts: they play a crucial role in studying various models of
statistical mechanics and probability such as $\beta$-ensembles and
generalizations of Selberg
integrals~\cite{OkounkovOlshanski1997,Kadell1997,Johansson1998,DumitriuEdelman2002,Mehta2004,For}. Furthermore,
they are strongly related to the Calogero--Sutherland model from
quantum mechanics \cite{LapointeVinet1995} and to random partitions
\cite{BorodinOlshanski2005,DolegaFeray2016,BorodinGorinGuionnet2017,Moll2015,DolegaSniady2019}. Finally,
they were found to have a rich combinatorial structure~\cite{Stanley1989,Macdonald1995,GouldenJackson1996,KnopSahi1997,ChapuyDolega2022,Moll2023} and its
understanding was emphasized as an important program in Okounkov's survey on random partitions~\cite{Okounkov2003}.

In his seminal work \cite{Stanley1989}, Stanley initiated the
combinatorial analysis of these symmetric functions. Knop and Sahi have given in \cite{KnopSahi1997} a combinatorial interpretation for the coefficients of the Jack polynomial $J_\lambda^{(\alpha)}$ in the monomial basis in terms of tableaux of shape $\lambda$.
A longstanding problem in combinatorics suggests that there is a
connection between the expansion of Jack polynomials in the power-sum
basis and the enumeration of discrete surfaces modeled by graphs
embedded in surfaces (also known as combinatorial
maps)~\cite{Hanlon1988,GouldenJackson1996,Lassalle2008b,Lassalle2009}. Besides
strong motivations coming from algebraic combinatorics, this expansion
is of particular interest due to the aforementioned connections
between Jack polynomials, probability, and mathematical physics
(see~\cite{DolegaFeray2016,BorodinGorinGuionnet2017} and references therein).

In this paper, we establish an explicit combinatorial expression of
Jack polynomials in the power-sum basis in terms of weighted maps,
answering a question posed by Hanlon~\cite{Hanlon1988} and proving its
more refined version conjectured by Lassalle
in~\cite{Lassalle2008b}. The study of the structure of Jack
polynomials, which underlies a one-parameter extension of Young's
formula for the irreducible characters of the symmetric group, has a
rich history of investigation that began with the work of
Hanlon~\cite{Hanlon1988} and had an interesting twist in recognizing
that mysterious enumerative properties emerge once the deformation
parameter is shifted by one. In the following, we provide a brief
account of this story.

\vspace{5pt}

It is well-known that the Jack polynomials coincide (up to a normalization factor) with the Schur polynomials at $\alpha = 1$. The power-sum expansion of the latter is given by the irreducible characters of the symmetric group. Consequently, Young's formula implies that for any Young diagram $\lambda \vdash n$, one can write:

\begin{equation}
  \label{eq:YoungsFormula}
  J_\lambda^{(\alpha=1)} = \sum_{\substack{\sigma_\circ \in
      \CS(T_\lambda),\\ \sigma_\bullet \in
      \RS(T_\lambda)}}(-1)^{|\sigma_\circ|}p_{\ct(\sigma_\circ\sigma_\bullet)},
\end{equation}
where $T_\lambda$ is a bijective filling of the Young diagram
$\lambda$ by the numbers $1,\dots,n$, the subgroups $\RS(T_\lambda), \CS(T_\lambda)
< \mathfrak{S}_{n}$ are the row and column stabilizers of $T_\lambda$,
 the
\emph{cycle-type} of $\sigma_\circ\sigma_\bullet$ is a partition $\ct(\sigma_\circ\sigma_\bullet)$ corresponding to the
lengths of cycles in $\sigma_\circ\sigma_\bullet$, and
$|\sigma_\circ|$ is the number of cycles of $\sigma_\circ$
(the fixed points are also counted). 

Inspired by Young's formula~\eqref{eq:YoungsFormula}, Hanlon asked
in~\cite{Hanlon1988} if there exists a function \sloppy $\stat\colon \RS(T_\lambda) \times \CS(T_\lambda) \to \ZZ_{\geq 0}$ such that

\begin{equation}
\label{eq:Hanlon}
J_\lambda^{(\alpha)} = \sum_{\substack{\sigma_\circ \in \CS(T_\lambda),\\ \sigma_\bullet \in \RS(T_\lambda)}} \alpha^{\stat(\sigma_\circ,\sigma_\bullet)}(-1)^{|\sigma_\circ|} p_{\ct(\sigma_\circ\sigma_\bullet)}.
\end{equation}

Since then, substantial progress has been made in understanding the
structure of Jack polynomials. Remarkably, it has been discovered by
different researchers in various
areas~\cite{GouldenJackson1996,ChekhovEynard2006a,Lassalle2008,Lassalle2009,AldayGaiottoTachikawa2010}
that once the original parametrization of Jack polynomials by $\alpha$
is replaced by its shifted version $b := \alpha-1$ (under different
disguises in the mentioned references), fascinating conjectural
enumerative properties emerge. In particular, it became more clear
that a one-parameter deformation of \eqref{eq:YoungsFormula}, which
seems less naive than \eqref{eq:Hanlon}, shall involve the parameter
$b$. In order to understand why, we rewrite~\eqref{eq:YoungsFormula}
in order to obtain a formula where we sum over bipartite orientable
maps (also known as bipartite ribbon graphs) which resembles the genus expansion known from random matrices.

\subsection{Young's formula in terms of maps}\label{ssec maps}

The study of maps is a well-developed area with strong connections with analytic combinatorics, mathematical physics and probability \cite{LandoZvonkin2004}. The relationship between generating series of maps and the theory of symmetric functions was first noticed via a character theoretic approach, and has then been developed to include other techniques such as matrix integrals and differential equations \cite{GouldenJackson1996,Lassalle2009,DolegaFeraySniady2014,ChapuyDolega2022}.

A \textit{connected map} is a connected graph embedded into a surface such that all the connected components of the complement of the graph are simply connected (see \cite[Definition 1.3.7]{LandoZvonkin2004}). These connected components are called the \textit{faces} of the map. We consider maps up to homeomorphisms of the surface. A connected map is \textit{orientable} if the underlying surface is orientable. In this paper\footnote{This is not the standard definition of a map; usually a map is connected.}, a \textit{map} is an unordered collection (possibly empty) of connected maps. A map is orientable if each of its connected components is orientable.

\begin{rmq}\label{rmq empty map}
By convention, we require that there are no isolated vertices in a map. In particular, the empty map is the only map with 0 edges.
\end{rmq}

We will restrict our attention to \textit{bipartite maps}, whose
vertices are colored black or white such that each edge connects
vertices of different colors. Note that in a bipartite map, all faces
have even degree. We define the \textit{face-type} of a bipartite map
$M$ with $n$ edges as the partition of $n$ obtained by reordering the
half degrees of the faces, and we denote it $\tf(M)$. We also denote
its set of white and black vertices by $\Vcirc(M)$ and $\Vbul(M)$,
respectively. It is well known in Hurwitz theory that orientable
bipartite maps with $n$ edges labeled by $1,\dots,n$ are in bijection
with pairs of permutations $(\sigma_\circ,\sigma_\bullet) \in
\mathfrak{S}_n$, and the cycle-type $\ct(\sigma_\circ\sigma_\bullet)$
corresponds to the face-type $\tf(M)$. The condition $\sigma_\circ \in
\CS(T_\lambda), \sigma_\bullet \in\RS(T_\lambda)$ can be interpreted
via this bijection as the possibility of embedding the associated map
$M$ into the Young diagram $\lambda$, as described
in~\cite{FeraySniady2011,DolegaFeraySniady2014}. These observations
are key steps in deriving the following formula from \eqref{eq:YoungsFormula}:
\begin{equation}
  \label{eq:StanleyFeray}
J^{(\alpha=1)}_\lambda=\frac{(-1)^n}{n!}\sum_{M}\sum_{f\colon \Vbul(M) \to [\ell(\lambda)]}p_{\tf(M)}\prod_{1\leq i\leq \ell(\lambda)}(-\lambda_{i})^ {|\Vcirc^{(i)}(M)|}.
\end{equation}
Here, we sum over bipartite orientable maps with $n:=|\lambda|$ edges
labelled by $\{1,2,\dots,n\}$, and $f$ is a coloring of black vertices
of $M$ such that the white vertices are colored by the maximal color
of their neighbors, and $|\Vcirc^{(i)}(M)|$ is the number of white
vertices of color $i$. An important property of this formula is that
it reveals a special structure of the irreducible characters of
the symmetric group, conjectured by
Stanley~\cite{Stanley2003/04}, which has very strong implications in
asymptotic representation theory. Estimating character
values for large symmetric groups is a very challenging problem:
Young's formula~\eqref{eq:YoungsFormula} does not seem to directly
apply to this problem,
and another classical tool such as Murnaghan--Nakayama rule quickly
becomes intractable when the size of the group is large, and can be used only in some specific settings~\cite{Roichman1996,MooreRussellSniady2007}. A very successful approach to the
asymptotic representation theory, initiated by Kerov and
Olshanski~\cite{KerovOlshanski1994}, treats normalized irreducible
characters as functions on Young diagrams, where the argument is
associated with the irreducible representation. Within this
approach a formula equivalent
to~\eqref{eq:StanleyFeray}, became the key ingredient in achieving a
breakthrough in asymptotic representation theory of the symmetric
groups~\cite{FeraySniady2011a}.

The approach of Kerov and
Olshanski was extended to the Jack case by Lassalle~\cite{Lassalle2008b,Lassalle2009}, where the primal object of study is \emph{the Jack character} $\Jch_\mu$. It is the function on Young diagrams defined by:
$$ \theta^{(\alpha)}_\mu(\lambda):=\left\{
    \begin{array}{ll}
       0 & \mbox{ if } |\lambda|<|\mu|,\\
        \binom{|\lambda|-|\mu|+m_1(\mu)}{m_1(\mu)}[p_{\mu,1^{n-m}}]J^{(\alpha)}_\lambda &\mbox{ if $|\lambda|\geq|\mu|$,}
    \end{array}
\right.$$
where $m_1(\mu)$ is the number of parts equal to 1 in the partition $\mu$. Stanley~\cite{Stanley2003/04} was studying the normalized irreducible characters of the symmetric group using this dual approach, and he observed that if one expresses them in different variables than $\lambda_1,\lambda_2,\dots$, which he called \emph{multirectangular coordinates}, then they have very special positivity and integrality properties.
\begin{defi}[\cite{Stanley2003/04}]
Let $k\geq1$ and let $s_1\geq s_2 \geq \dots \geq s_k\geq 1$ and $r_1,\dots r_k$ be two sequences of non negative integers.
We say that $(s_1,s_2,\dots,s_k)$ and $(r_1,\dots,r_k)$ are \textit{multirectangular coordinates} for a partition $\lambda$ and we  denote 
$\lambda=\bfs^\bfr$, if $\lambda$ is the union of $k$ rectangles of sizes $s_i\times r_i$, or equivalently  $\lambda=[s_1^{r_1}\dots s_k^{r_k}]$, see \cref{fig multirectangular coordinates} for an example. 
\end{defi}
 
 \begin{figure}
     \centering
     \includegraphics[width=0.27\textwidth]{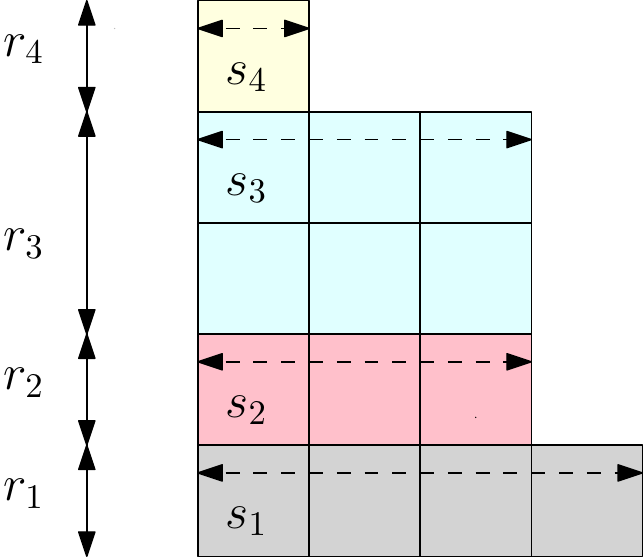}
     \caption{The Young diagram of the partition $[4,3,3,3,1]$ as the union of 4 rectangles, with $\bfs=(4,3,3,1)$ and $\bfr=(1,1,2,1)$ as multirectangular coordinates.}
      \label{fig multirectangular coordinates}
    \end{figure}

Since we do not require that the sequence $\bfs$ be strictly decreasing, the multirectangular coordinates are not unique in general. If $\lambda$ is a partition of multirectangular coordinates $(s_1,\dots,s_k)$ and $(r_1,\dots, r_k)$, we write, for any partition $\mu$,
$$\tJch_\mu(\bfs,\bfr):=\Jch_\mu(\lambda),$$
where $\bfr=(r_1,\dots,r_k,0\dots)$ and $\bfs=(s_1,\dots,s_k,0,\dots).$
Stanley found in~\cite{Stanley2003/04} an explicit formula for
$\theta^{(\alpha=1)}_\mu(\lambda)$ when $\lambda$ is a rectangle, and
he conjectured a formula for general $\bfr,\bfs$, which implies that
the normalized irreducible character $(-1)^{|\mu|}z_\mu
\tilde{\theta}^{(\alpha=1)}_\mu(\bfs,\bfr)$ is a polynomial in the variables
$-s_1,-s_2,\dots,r_1,r_2,\dots$ with {\bf non-negative integer
coefficients}, and the aforementioned consequences in asymptotic
representation theory follow. This formula, nowadays known as F\'eray--Stanley formula, is equivalent
to~\eqref{eq:StanleyFeray}, and it was soon after proved by F\'eray~\cite{Feray2010}, and then reproved in~\cite{FeraySniady2011a}.

Lassalle revisited in~\cite{Lassalle2008b} the question of
Hanlon regarding the combinatorial structure of Jack characters, with
an insight of F\'eray--Stanley that explains this structure when
$\alpha=1$. He conjectured that this special case of $\alpha=1$ is in fact a shadow of a more general phenomenon that holds for arbitrary Jack characters, once the parameter $\alpha$ is replaced by the parameter $b:=\alpha-1$:
\begin{conj}\label{Lassalle conj}
The normalized Jack characters expressed in the Stanley coordinates \sloppy$(-1)^{|\mu|}z_\mu\tJch_\mu(\bfs,\bfr)$ are polynomials in the variables $b,-s_1,-s_2,\dots,r_1,r_2,\dots$ with {\bf non-negative integer
coefficients}, where $b := \alpha-1$.
\end{conj}

Lassalle speculated that this reparametrization reflects a true
combinatorial meaning of Jack characters, but he was unable to find
its concrete, even conjectural interpretation. A formula analogous
to~\eqref{eq:StanleyFeray} was found in the case
$\alpha=2$~\cite{FeraySniady2011}, establishing a special case of
\cref{Lassalle conj}, next to the classical $\alpha=1$ case. Some other
very special cases were proven
in \cite{DolegaFeraySniady2014,BenDali2022a}, but despite many efforts,
\cref{Lassalle conj} remains unproven for the last 15 years. The
understanding of special cases $\alpha=1,2$ heavily relies on
representation-theoretical techniques commonly used to study Schur or
Zonal polynomials. These tools do not seem to exist
for general $\alpha$, and any attempt of attacking~\cref{Lassalle
  conj} requires developing new methods.

\vspace{5pt}

In this paper, we answer Hanlon's question about the combinatorial
structure of Jack polynomials expressed in the power-sum basis, for
which we prove an explicit formula. Actually, we prove a more general
result by giving a combinatorial interpretation for Jack
characters. Consequently, we prove \cref{Lassalle conj} by giving a
meaning to the parameter $b$ that interpolates between the generating
function of graphs drawn on orientable surfaces (for $b=0$) and on
non-oriented surfaces (for $b=1$). The main tool to prove our results
is a differential calculus approach developed by the second author and
Chapuy in~\cite{ChapuyDolega2022}. We combine this approach with an
algebraic characterization of Jack characters as the unique shifted
symmetric functions determined by vanishing conditions in the spirit
of~\cite{KnopSahi1996}, and we furthermore rely on combinatorial
aspects of the integrable system of
Nazarov--Sklyanin~\cite{NazarovSklyanin2013}, building new connections
between different fields that implicitly involve Jack polynomials. We
now describe our results in more details.

\subsection{First main result: combinatorial formula for Jack characters}

As mentioned earlier, our formula interprets Jack characters as
generating functions that interpolate between orientable and
non-oriented maps, and we use the following definition, introduced by
Goulden and Jackson in the context of the $b$-conjecture.

\begin{defi}[\cite{GouldenJackson1996}]\label{def SON}
\textit{A statistic of non-orientability} on bipartite maps is a statistic $\vartheta$ with non-negative integer values, such that $\vartheta(M)=0$ if and only if $M$ is orientable.
\end{defi}

In practice, a statistic of non-orientability is used to "quantify" the non-orientability of a map by counting the number of edges which contribute to its non-orientability, following a given algorithm of decomposition of the map. 
Several examples of such statistics have been introduced in prior
works
\cite{LaCroix2009,DolegaFeraySniady2014,Dolega2017a,ChapuyDolega2022}. We
are ready to introduce the main character in our story.

\begin{defi}\label{def layered maps}
Fix $k\geq0$. We say that a bipartite map $M$ is $k$-layered if its vertices are partitioned into $k$ sets (which may be empty), called the layers of the map; $\Vcirc(M)=\bigcup_{1\leq i\leq k}\Vcirc^{(i)}(M)$ (resp. $\Vbul(M)=\bigcup_{1\leq i\leq k}\Vbul^{(i)}(M)$), which satisfies the following condition: if $v$ is a white vertex in a layer $i$, then all its neighbors are in layers $j\leq i$, and it has at least one neighbor in the layer $i$.

For $1\leq i\leq k$, we define the partition $\nu_\bullet^{(i)}(M)$ obtained by ordering the degrees of the black vertices in the layer $i$.
A $k$-layered map is labelled if:
\begin{itemize}
    \item in each layer $1\leq i\leq k$, the black vertices having the same degree $j$ are numbered by $1,2,\dots, m_j\left(\nu_\bullet^{(i)}(M)\right)$,
    \item each black vertex has a distinguished oriented corner (a
      small neighborhood around the vertex delimited by two
      consecutive edges).
\end{itemize}
\end{defi}
An example of a 2-layered map is given in \cref{fig. 2-layer map}.
\begin{figure}[t]
    \centering
    \includegraphics[width=0.3\textwidth]{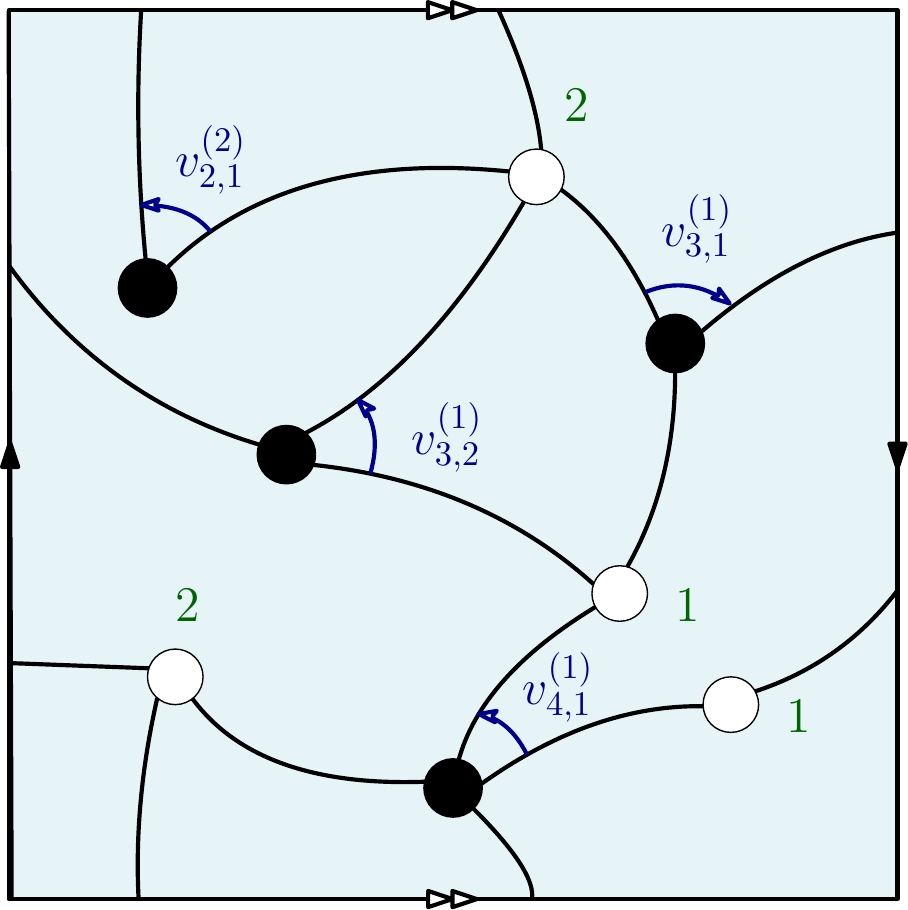}
    \caption{A 2-layered map on the Klein bottle, represented here by a square whose left side should be glued to the right one (with a
twist) and the top side should be glued to the bottom one (without a twist), as
indicated by the arrows. Moreover, $v^{(i)}_{j,m}$ denotes the black vertex of degree $j$ numbered by $m$ in the layer $i$, and the integer next to a white vertices indicate the number of its layer.}
    \label{fig. 2-layer map}
\end{figure}
Note that a $k$-layered map can be seen as $(k+1)$-layered map with an empty layer $k+1$.
We call a \textit{layered map} a $k$-layered map for some $k\geq 1$.
We denote by $\Mk$ (resp. $\Minf$) the set of all labelled $k$-layered
maps (resp. labelled layered maps). Similarly, we denote $\Mk_\mu$ and
$\Minf_\mu$ those of face type $\mu$.

From now on all the layered maps will be labelled (unless stated otherwise).

\begin{rmq}\label{rmq def Stanley}
This definition of layered maps is equivalent to maps equipped with a
coloring such as in
\eqref{eq:StanleyFeray}, cf.~\cite[Section 1.6]{DolegaFeraySniady2014}. However, we prefer to present this definition as above since it will play a slightly different role in this paper.
\end{rmq}

\begin{thm}[The First Main Result]\label{thm Jack char}
There exists a statistic of non-orientability $\vartheta$ on layered maps such that for any partitions $\mu$ and $\lambda$, we have
\begin{equation}\label{eq main thm}
\theta_\mu^{(\alpha)}(\lambda)=(-1)^{|\mu|}\sum_{M\in\Minf_\mu}\frac{b^{\vartheta_\rho(M)}}{2^{|\Vbul(M)|-\cc(M)}\alpha^{cc(M)}}\prod_{i\geq 1}\frac{(-\alpha\lambda_{i})^{|\mathcal{V}_\circ^{(i)}(M)|}}{z_{\nu_\bullet^{(i)}(M)}}
\end{equation}
where $b$ is the parameter related to $\alpha$ by $b:=\alpha-1$ and
$z_{\nu_\bullet^{(i)}(M)}$ is the classical normalization factor (see
\cref{subsec Partitions}).
\end{thm}

Note that the product and the sum here are finite since each layered map has a finite number of non-empty layers. Actually, we prove that this theorem holds for a family of statistics $\vartheta$. Although the two parameters $\alpha$ and $b$ are related, we prefer to keep both of them in the previous formula since they play different roles. In particular, one may notice that the quantity $1/\left(2^{(|\Vbul(M)|-\cc(M))}\alpha^{cc(M)}\right)\prod_{1\leq i\leq k}\frac{(-\alpha\lambda_{i})^{|\mathcal{V}_\circ^{(i)}(M)|}}{z_{\nu_\bullet^{(i)}(M)}}$ depends only on the underlying "$k$-layered graph" of $M$ and is rather straightforward. This will not be the case for the quantity $b^{\vartheta(M)}$, called the \textit{$b$-weight} of the map.

In the case $b=0$, and by definition of a statistic of non-orientability, only bipartite maps appear in \cref{eq main thm}, so that it recovers Féray--Stanley formula~\eqref{eq:StanleyFeray}. Similarly, \cref{eq main thm} coincides when $b=1$ with the expression given in \cite[Thm 1.2]{FeraySniady2011}.

As a direct consequence of \cref{thm Jack char}, we obtain the following interpretation of Jack polynomials in the basis of power-sum functions.
\begin{thm}\label{thm Jack}
Let $n$ be a positive integer and let $\lambda$ be a partition of $n$. Then 
$$J^{(\alpha)}_\lambda=(-1)^n\sum_{M}p_{\tf(M)}\frac{b^{\vartheta_\rho(M)}}{2^{|\Vbul(M)|-\cc(M)}\alpha^{cc(M)}}\prod_{1\leq i\leq \ell(\lambda)}\frac{(-\alpha \lambda_{i})^{|\mathcal{V}_\circ^{(i)}(M)|}}{z_{\nu_\bullet^{(i)}(M)}},  
$$
where the sum is taken over all $\ell(\lambda)$-layered maps $M$ with $n$ edges.
\end{thm}

As in \cref{eq:StanleyFeray}, the sum in \cref{thm Jack} can be
interpreted using maps whose edges are embedded in a non-bijective way
in the Young diagram of $\lambda$ (see also \cref{rmq def
  Stanley}). Hence, our second main result can be considered as a
non-bijective version of Hanlon's conjecture\footnote{It is worth
  mentioning that Haglund and Wilson have found recently a very different
  combinatorial formula for the expansion of Jack polynomials in the power-sum
  basis~\cite[Corollary 4.3.1]{HaglundWilson2020}. The combinatorial
  objects in their formula are not weighted by a simple monomial in $\alpha$/$b$,
  but by more complicated products of $\alpha$-deformations of hooks,
  and the classical Young's formula~\eqref{eq:YoungsFormula} is not
  immediately obtained as the special case.}.

As we have already mentioned, the ideas used in the proofs of the special cases $b=0$ and $b=1$ of \cref{thm Jack char} do not apply to the general case, and we need to develop a new approach. Our strategy is to use a characterization of Jack characters $\theta_\mu^{(\alpha)}$ given by~\cref{thm Feray} (due to Féray~\cite[Theorem A.2]{Sniady2015a}) and prove that the right-hand side of \cref{eq main thm} satisfies the conditions of \cref{thm Feray}, see \cref{thm vanishing} and \cref{thm symmetry}. 

 A key technique
for proving these theorems, is using differential operators to study the properties of the generating series of layered maps. 
We prove various commutation relations between them using methods
inspired by the theory of Lie algebras, and we show that these algebraic relations reflect the desired combinatorial and algebraic properties of the right-hand side in \cref{thm Jack char}.

\subsection{Lassalle's conjecture and second main result}\label{ssec Lassalle conj}

Our second main result is the proof of \cref{Lassalle conj}.

\begin{thm}\label{thm Lassalle conj}
The normalized Jack characters expressed in the Stanley coordinates \sloppy$(-1)^{|\mu|}z_\mu\tJch_\mu(\bfs,\bfr)$ are polynomials in the variables $b,-s_1,-s_2,\dots,r_1,r_2,\dots$ with non-negative integer coefficients, where $b := \alpha-1$.
\end{thm}

Note that \cref{thm Lassalle conj} does not immediately follow
from~\cref{thm Jack char}: the integrality part is not
obvious, and requires new ideas. Indeed, the proof of \cref{thm Lassalle conj} consists of two parts that are
proved using very different techniques. In the first part we deduce
positivity as a consequence of the combinatorial expression of Jack
characters obtained in \cref{thm Jack char}. In the second part, we
obtain the integrality using integrable system of
Nazarov--Sklyanin~\cite{NazarovSklyanin2013}. We relate their theory
with Jack characters by proving an explicit combinatorial formula
expressing a certain basis of shifted symmetric functions in terms of
normalized Jack characters. We conclude by showing that the transition
matrix between these two bases is invertible over $\mathbb{Z}$. As a byproduct, we prove that Kerov polynomials for Jack characters have integer coefficients, which was an open problem posed by Lassalle in~\cite{Lassalle2009} (see~\cref{sec:Int} for details).

\subsection{Third main result: Jack polynomials via differential operators}
Another application of the first main result is a formula for the
expansion of Jack polynomials in the power-sum basis using creation
operators (see also \cref{thm Jack via diff op 2}). 
\begin{thm}\label{thm Jack via diff op}
Fix a partition $\lambda = (\lambda_1,\dots,\lambda_\ell)$. Then
\begin{equation*}
    J^{(\alpha)}_\lambda
      =\B^{(+)}_{\lambda_1}\cdots \B^{(+)}_{\lambda_\ell} \cdot 1,
    \end{equation*}
    where $\B^{(+)}_n := [t^{n}]\exp\left(\Binf(-t,\bfp,-\alpha n)\right)$.
\end{thm}
\noindent One may notice that this formula is simpler than the one
obtained by combining \cref{thm Jack} and \cref{prop expr F}. In fact,
it is obtained from these two results using some properties of the
differential operators proved in \cref{sec vanishing}  (see~\cref{ssec
  Jack via diff} for more details).

\subsection{Outline of the paper}
The paper is organized as follows. In \cref{sec preliminaries} we
introduce some notation related to partitions and symmetric functions
and we give a proof of~\cref{thm Feray} that characterizes the Jack characters. In \cref{sec:Int}, we prove the integrality in \cref{thm Lassalle conj}. In \cref{sec comb model} we explain the combinatorial decomposition of layered maps and we give a differential expression for the generating series of layered maps.  \cref{sec vanishing} is dedicated to the proof of the first characterization property, namely the vanishing property. In \cref{sec com rel} we prove a series of commutation relations for differential operators which are used to obtain the second characterization property in \cref{sec shifted sym prop}. In \cref{sec proof first main result}, we finish the proof of Theorem \ref{thm Jack char} and \ref{thm Jack via diff op} and we prove the positivity in \cref{thm Lassalle conj}.

\section{Notation and preliminaries}\label{sec preliminaries}

For the definitions and notation introduced in \cref{subsec Partitions,sec SymFun} 
we refer to \cite{Stanley1989,Macdonald1995}.
\subsection{Partitions}\label{subsec Partitions}

A \textit{partition} $\lambda=[\lambda_1,...,\lambda_\ell]$ is a weakly decreasing sequence of positive integers $\lambda_1\geq...\geq\lambda_\ell>0$. We denote by $\mathbbm{Y}$ the set of all integer partitions. The integer $\ell$ is called the \textit{length} of $\lambda$ and is denoted $\ell(\lambda)$. The size of $\lambda$ is the integer $|\lambda|:=\lambda_1+\lambda_2+...+\lambda_\ell.$ If $n$ is the \textit{size} of $\lambda$, we say that $\lambda$ is a partition of $n$ and we write $\lambda\vdash n$. The integers $\lambda_1$,...,$\lambda_\ell$ are called the \textit{parts} of $\lambda$. For $i\geq 1$, we denote $m_i(\lambda)$ the number of parts of size $i$ in $\lambda$. We set then 
$$z_\lambda:=\prod_{i\geq1}m_i(\lambda)!i^{m_i(\lambda)}.$$ 
We denote by $\leq$ the \textit{dominance partial} ordering on partitions, defined by 
$$\mu\leq\lambda \iff |\mu|=|\lambda| \text{ and }\hspace{0.3cm} \mu_1+...+\mu_i\leq \lambda_1+...+\lambda_i \text{ for } i\geq1.$$

\noindent We identify a partition  $\lambda$ with its \textit{Young diagram}, defined by 
$$\lambda:=\{(i,j),1\leq i\leq \ell(\lambda),1\leq j\leq \lambda_i\}.$$
\textit{The conjugate partition} of $\lambda$, denoted $\lambda^t$, is the partition associated to the Young diagram obtained by reflecting the diagram of $\lambda$ with respect to the line $j=i$:
$$\lambda^t:=\{(i,j),1\leq j\leq \ell(\lambda),1\leq i\leq
\lambda_i\}.$$

Finally, we define the \textit{$\alpha$-content} of a box $\Box:=(i,j)$ by 

\begin{equation}\label{eq alpha content}
    c_\alpha(\Box):=\alpha(j-1)-(i-1).
\end{equation}

\subsection{Symmetric functions and Jack polynomials}\label{sec SymFun}

We fix an alphabet $\mathbf{x}:=(x_1,x_2,..)$. We denote by $\mathcal{S}$ the algebra of symmetric functions in $\mathbf{x}$ with coefficients in $\mathbb Q$. For every partition $\lambda$, we denote $m_\lambda$ the monomial function and $p_\lambda$ the power-sum function associated to the partition $\lambda$.
We  consider the associated alphabet of power-sum functions
$\mathbf{p}:=(p_1,p_2,..)$. 

Let $\mathcal{S}_\alpha$ be the algebra of symmetric functions with
coefficients in $\mathbb{Q}(\alpha)$. It is well known that the
power-sum functions form a basis of the symmetric functions algebra,
therefore $\mathcal{S}_\alpha $ can be identified with the polynomial
algebra $\mathbb{Q}(\alpha))[\pp]$. In particular the following operator acts on
$\mathcal{S}_\alpha$:
\begin{equation}
\label{eq:Laplace-Beltrami}
   D_\alpha = \frac{1}{2}\left(\alpha\sum_{i,j\geq 1}p_{i+j}\frac{ij\partial^2}{\partial
  p_i \partial p_{j}} +\sum_{i,j\geq
  1}p_{i}p_j\frac{(i+j)\partial}{\partial p_{i+j}}+(\alpha-1)\cdot\sum_{i\geq
  1}p_{i}\frac{i(i-1)\partial}{\partial p_i}\right).
\end{equation}
This operator, called the \emph{Laplace-Beltrami operator}, can be thought
of as a defining operator for the\emph{Jack symmetric
functions}.


\begin{defprop}[Definition-Proposition 2.1 in~\cite{ChapuyDolega2022}]
  \label{defprop:Jack}
There is a unique family of symmetric functions
$(J^{(\alpha)}_\lambda)_{\lambda\in\mathbb Y}$ such that for each partition $\lambda$,
\begin{itemize}
\item $D_\alpha J^{(\alpha)}_\lambda = \left(\sum_{\square \in \la}c_\a(\square)\right)J^{(\alpha)}_\lambda$;
\item
$ J_\la^{(\a)} = \prod_{(i, j)\in\lambda}\big(\alpha(\lambda_i-j)+(\lambda_j^t-i)+1\big) m_\la + \sum_{\nu < \la}a^{\la}_\nu m_\nu,  \text{ where } a^{\la}_\nu \in \QQ(\a).$
  \end{itemize}
  We call them \textbf{Jack symmetric functions}.
\end{defprop}

We will often use the following quantity:
\begin{equation}\label{eq j alpha}
     j_\lambda^{(\alpha)} :=\prod_{(i, j)\in\lambda}\big(\alpha(\lambda_i-j)+(\lambda_j^t-i)+1\big) \big(\alpha(\lambda_i-j)+(\lambda_j^t-i)+\alpha\big).
 \end{equation}
In this paper, Jack polynomials will always be expressed in the
power-sum variables $\bfp$ rather than the alphabet $\mathbf{x}$. We will also often work with the substitution
\[J_\lambda^{(\alpha)}(\underline{u}) :=
  J_\lambda^{(\alpha)}\big|_{p_1=p_2=\cdots = u},\]
where $u$ is a variable. The following theorem due to
Macdonald \cite[Chapter VI Eq. 10.25]{Macdonald1995} gives an
expression of $J_\lambda^{(\alpha)}(\underline{u})$. 
\begin{thm}[\cite{Macdonald1995}]\label{thm Jack formula}
For every $\lambda\in\mathbbm{Y}$, we have
$$J_\lambda^{(\alpha)}(\underline{u})=\prod_{\Box\in\lambda}\left(u+c_\alpha(\Box)\right).$$
\end{thm}

\subsection{Lassalle's isomorphism and characterization of Jack characters}
The purpose of this section is to reprove a theorem of F\'eray that
uniquely determines Jack characters as shifted symmetric functions
with specific vanishing properties. 
We start by recalling some results on shifted symmetric functions from \cite{Lassalle2008b}. Several of these results were based on the work of Knop and Sahi \cite{KnopSahi1996}.
\begin{defi}[\cite{Lassalle2008b}]\label{ssec shifted fcts}
We say that a polynomial of degree $n$ in $k$ variables
$(s_1,\dots,s_k)$ with coefficients in $\mathbb Q (\alpha)$ is
\textit{$\alpha$-shifted symmetric} of degree $n$ if it is symmetric in the variables $s_i-i/\alpha$.
An $\alpha$-shifted symmetric function of degree $n$ is a sequence
$(f_k)_{k\geq1}$ such that for every $k\geq 1$, the function $f_k$ is
an $\alpha$-shifted symmetric polynomial of degree $n$ in $k$ variables and 
\begin{equation}\label{eq shifted functions}
  f_{k+1}(s_1,\dots,s_k,0)=f_k(s_1,\dots,s_k).  
\end{equation}
We denote by $\Shifted$ the algebra of $\alpha$-shifted symmetric functions.
\end{defi}

Let $f$ be an $\alpha$-shifted symmetric function and $\lambda =
(\lambda_1,\dots,\lambda_k)$ a partition. Then we denote
$f(\lambda):=f(\lambda_1,\dots, \lambda_{k},0,\dots)$. It turns that $\alpha$-shifted function is completely determined by its evaluation on partitions $(f(\lambda))_{\lambda\in\mathbb{Y}}$. Lassalle \cite{Lassalle2008b} has constructed an isomorphism
$f\longmapsto f^\#$ between $\Sym$ and $\Shifted$, which satisfies the following properties
\begin{enumerate}
    \item If $f$ is homogeneous, then the top degree
      part of $f^\#$ is $f$.
    \item For any partition $\xi$, the function $\Jxish$ is the unique
      $\alpha$-shifted symmetric function such that $\Jxish(\xi)\neq0$
      and $\Jxish(\lambda)=0$ if $\lambda$ does not contain $\xi$.
\end{enumerate}

Moreover, the image of power-sum functions by this isomorophism are the Jack characters up to explicit factors.
\begin{lem}[\cite{Lassalle2008b}, Proposition 2.9 in \cite{DolegaFeray2016}]\label{lem pdiez}
For any partition $\mu$, we have $\alpha^{|\mu|-\ell(\mu)}/z_\mu\cdot p_\mu^\#=\Jch_\mu.$
\end{lem}

Since power-sum functions $(p_\mu)_{\mu\in\mathbb{Y}}$ form a linear
basis of $\mathcal{S}_\alpha$, \cref{lem pdiez} implies that Jack
characters $(\Jch_\mu)_{\mu\in\mathbb{Y}}$ form a basis of $\mathcal
S^*_\alpha$. Similarly, $(J^{(\alpha)\#}_\mu)_{\mu\in\mathbb{Y}}$ is a
basis for $\mathcal S^*_\alpha$.

The starting point of
the proof of the first main theorem is the following
characterization of the Jack characters $\theta_\mu^{(\alpha)}$ proved
by F\'eray (see
\cite[Theorem A.2]{Sniady2015a}). For completeness, we give its proof.
\begin{thm}\label{thm Feray}
Fix a partition $\mu$. The Jack character $\Jch_\mu$ is the unique $\alpha$-shifted symmetric function of degree $|\mu|$ with top homogeneous part $\alpha^{|\mu|-\ell(\mu)}/z_\mu\cdot p_\mu$, such that $\Jch_\mu(\lambda)=0$ for any partition $|\lambda|<|\mu|$.
\end{thm}

We now prove \cref{thm Feray}.
\begin{proof}[Proof of \cref{thm Feray}]
The fact that $\Jch_\mu(\lambda)=0$ if $|\lambda|<|\mu|$ comes from the definition. Its  top homogeneous part is obtained from property $(1)$ and \cref{lem pdiez} above.

\textit{Uniqueness:} Let $G$ be an $\alpha$-shifted symmetric function of degree $|\mu |$ with the same top degree part as $\Jch_\mu$, and such that $G(\lambda)=0$ for any $|\lambda |<|\mu|$. Set $G:=F-\Jch_\mu$. Then $G$ is an $\alpha$-shifted symmetric function of degree at most $|\mu|-1$ with 
\begin{equation}\label{eq G van}
    G(\lambda)=0 \text{ for $|\lambda|<|\mu|$}. 
\end{equation}
We expand $G$ in the $\Jxish$ basis 
\begin{equation}\label{eq G J}
  G=\sum_\xi c_\xi \Jxish.  
\end{equation}

As $\deg(G)\leq |\mu|-1,$ the sum can be restricted to partitions $\xi$ of size at most $|\mu|-1$. 
We will prove by contradiction that $G=0$, \textit{i.e} that $c_\xi=0$
for all partitions $\xi$ with $|\xi|\leq |\mu|-1$. Assume this is not
the case and consider a partition $\xi_0$ of minimal size such that
$c_{\xi_0}\neq 0$.

\cref{eq G van} gives $G(\xi_0)=0$ since $|\xi_0|<|\mu|$. On the other
hand, $\Jxish(\xi_0)=0$ if $\xi_0$ does not contain $\xi$ (see
property (2) above). Therefore the RHS of \cref{eq G J} evaluated on
$\xi_0$ vanishes for
all partitions $\xi$ except for $\xi=\xi_0$. Moreover,
$c_{\xi_0}\neq 0$ by the assumptions and
$J^{(\alpha)\#}_{\xi_0}(\xi_0)\neq0$ from property (2). Therefore
$G(\xi_0) = c_{\xi_0} J^{(\alpha)\#}_{\xi_0}(\xi_0) \neq 0$, and we have reached a contradiction. Hence, $G=0$ and the uniqueness is proved.
\end{proof}

\section{Integrality in Lassalle's conjecture}\label{sec:Int}

Before we prove \cref{thm Jack char} we present the proof of \cref{thm
  Lassalle conj}. The positivity part follows directly from the
combinatorial interpretation (in terms of layered maps) stated in
\cref{thm Jack char}, whose proof is technically involved and will
occupy the most part of this paper. Integrality, however, does not
follow from \cref{thm Jack char} and requires new ideas. We prove it using different approach based on combinatorics of Nazarov--Sklyanin operators interpreted as lattice paths. These developments are independent of the other sections, and they are also of independent interest, as we demonstrate by proving other problems stated in the literature as a byproduct.

\subsection{Nazarov--Sklyanin operators and $\alpha$-polynomial functions}

Recall that $\theta^{(\alpha)}_\mu$ is a linear basis of the algebra $\Shifted$ of $\alpha$-shifted symmetric functions. It turns that a strictly related algebra is of special interest. Let $\gamma := \sqrt{\alpha}^{-1}-\sqrt{\alpha}$ and define the \emph{normalized Jack character}
\begin{equation}
  \label{eq:JackCharacter}
  \Cha_\mu(\lambda) := \alpha^{\frac{\ell(\mu)-|\mu|}{2}} z_\mu \theta^{(\alpha)}_\mu(\lambda),
\end{equation}
and the $\QQ[\gamma]$-module $\Poly$ spanned by
$\left(\Cha_\mu\right)_{\mu \in \YY}$. It was proved that $\Poly$ is
in fact an algebra called the algebra of \emph{$\alpha$-polynomial
  functions}, see~\cite{DolegaFeray2016,Sniady2019,DolegaSniady2019}
for the details. This normalization for Jack characters is of a special interest due to the connections with random partitions. In fact, several important bases of $\Poly$ grew up from this connection that we describe now.

\subsubsection{Kerov's transition measure}

Kerov associated with a Young diagrams $\lambda$ certain probability measure $\mu_{\lambda}$ on $\RR$ that is very useful for studying asymptotic behaviour of random Young diagrams. This \emph{transition measure} is uniquely characterized by its Cauchy transform:
\begin{equation}
  \label{eq:Cauchy-St}
  G_{\mu_\lambda}(z) := \int_\RR \frac{d\mu_\lambda(x)}{z-x} =
  \frac{1}{z+\ell(\lambda)}\prod_{i=1}^{\ell(\lambda)}\frac{z+i-\lambda_i}{z+i-1-\lambda_i}.
\end{equation}
In particular the $\ell$-th moment $M_\ell(\lambda)$ of the transition measure $\mu_\lambda$ can be computed by applying a simple relation between the Cauchy transform expanded around infinity and the generating function of moments:
\[ z^{-1}+\sum_{\ell \geq 1}M_\ell(\lambda)z^{-\ell-1} = G_{\mu_\lambda}(z).\]
Note that $M_\ell(\lambda)$ can be treated as functions of $\lambda_1,\dots,\lambda_{\ell(\lambda)}$, and define
\begin{equation}
  \label{eq:moments}
  M^{(\alpha)}_{\ell}(\lambda) := \alpha^{-\frac{\ell}{2}}M_\ell(\alpha \cdot \lambda_1,\dots,\alpha\cdot\lambda_{\ell(\lambda)}),
\end{equation}
which is a well-defined function on $\mathbb{Y}$.
It was proved in~\cite{DolegaFeray2016} that $M^{(\alpha)}_{\ell}$ is an algebraic basis of $\Poly$.

\begin{thm}
  \label{theo:MBasis}
  The algebra $\Poly$ is generated (over $\QQ[\gamma]$) by $\left(M^{(\alpha)}_{\ell}\right)_{\ell \geq 2}$.  
\end{thm}

The above theorem is a starting point for defining other interesting bases using other observables arising from classical and free probability. Besides the \emph{moments}, we will use the \emph{Boolean cumulants} $B^{(\alpha)}_{\ell}$,
and the \emph{free cumulants} $R^{(\alpha)}_{\ell}$. In our context it would be the most convenient to define them by the following recursive formulas that can be easily inverted over $\ZZ$, (see~\cite[Proposition 2.2]{DolegaFeraySniady2010} and \cite[Proposition 3.2]{CuencaDolegaMoll2023}).

\begin{prop}
  \label{prop:Mom-Bool-Free}
  For any integer $\ell \geq 2$,
  \begin{align}
    M^{(\alpha)}_\ell &= \sum_{n \geq 1}\sum_{k_1,\dots, k_n \geq
    2\atop k_1+\cdots k_n = \ell}B^{(\alpha)}_{k_1}\cdots B^{(\alpha)}_{k_n},\label{eq:Moment-Boolean}\\
        M^{(\alpha)}_\ell &= \sum_{n \geq 1}\frac{(\ell)_{n-1}}{n!}\sum_{k_1,\dots, k_n \geq
    2\atop k_1+\cdots k_n = \ell}R^{(\alpha)}_{k_1}\cdots R^{(\alpha)}_{k_n}.\label{eq:Moment-Free}
    \end{align}
  \end{prop}

  In particular we have the following theorem

  \begin{thm}
  \label{theo:XBasis}
  The algebra $\Poly$ is generated (over $\QQ[\gamma]$) by $\left(X^{(\alpha)}_{\ell}\right)_{\ell \geq 2}$, where $X = M,B,R$. 
\end{thm}

\subsubsection{Nazarov--Sklyanin operators}

Consider the (infinite) row vector
$P = (P_{1,{k}})_{{k} \in \NN_{\geq 1}}$ and dually the column vector $P^\dagger =
(P^\dagger_{{k},1})_{{k} \in \NN_{\geq 1}}$, where $P_{1,{k}} := \sqrt{\alpha}^{-1}\cdot p_{{k}}$, and
$P^\dagger_{{k}, 1} := \sqrt{\alpha}\cdot  k\cdot\frac{\partial}{\partial p_{k}} =: \sqrt{\alpha}^{-1}\cdot p_{{-k}}$, are regarded as operators on the algebra of symmetric
functions $\Sym$.
Let $L = (L_{i,j})_{i,j \in \NN_{\geq 1}}$ be the infinite matrix defined by
$L_{i,j} := \sqrt{\alpha}^{-1}\cdot p_{j-i}-\delta_{i,j}\,i\gamma$, for all $i, j\in\NN_{\geq 1}$, with the convention that
$p_0:=0$:
$$
L = \begin{bmatrix} 
-\gamma & P_{1,{1}} & P_{1,{2}} & \cdots\\
P^\dagger_{{1},1} & -2\gamma & P_{1,{1}} & \ddots\\
P^\dagger_{{2},1} &P^\dagger_{{1},1} & -3\gamma & \ddots\\
\vdots & \ddots & \ddots & \ddots 
\end{bmatrix}.
$$
The main
result of Nazarov--Sklyanin~\cite[Theorem 2]{NazarovSklyanin2013} can be
reformulated as follows:

\begin{thm}
  \label{theo:Nazarov-Sklyanin}
  The following equality holds true for all $\ell\geq 0$:
  \begin{align}
        PL^\ell P^\dagger \Jla &=
    B^{(\alpha)}_{\ell+2}(\lambda) \cdot \Jla, \label{eq:NS-Boolean}
    \end{align}
  \end{thm}

  \begin{rmq}
    Note that the Laplace--Beltrami operator $D_\alpha$ equals to
    $\frac{\sqrt{\alpha}}{2}PL P^\dagger 
    $.
    \end{rmq}

Nazarov and Sklyanin stated their theorem differently, as they did not realize the connection with the transition measure, the fact which is crucial for us. This connection was first noticed by Moll~\cite{Moll2015}, and we refer the reader to the proof of \cref{theo:Nazarov-Sklyanin} presented in~\cite[Theorem 3.9]{CuencaDolegaMoll2023}.

\subsection{Integrality}

\cref{theo:XBasis} implies that for every $\ell_1,\dots,\ell_k \geq 2$ and for every $X = B,M,R$ there exists a polynomial $X_{\ell_1,\dots,\ell_k}(x_0,x_1,x_2,\dots)$ with rational coefficients such that
\[ X^{(\alpha)}_{\ell_1}\cdots X^{(\alpha)}_{\ell_k} = X_{\ell_1,\dots,\ell_k}(-\gamma,x_1,x_2,\dots)\]
with the convention that we identify the monomials with the normalized Jack characters $x_1^{m_1}\cdot x_2^{m_2}\cdots = \Ch^{(\alpha)}_{1^{m_1},2^{m_2},\dots}$. We prove that for $X = B,M$ the polynomials $X_{\ell_1,\dots,\ell_n}$ have non-negative integer coefficients and we provide their combinatorial interpretation in terms of Łukasiewicz ribbon paths introduced in~\cite{CuencaDolegaMoll2023}.

\subsubsection{Łukasiewicz ribbon paths}

Informally speaking, an \emph{excursion} is a directed lattice path
with steps of the form $(1,k), k \in\ZZ$, starting at
$(0,0)$, finishing at $(\ell,0)$, and that stays in the first
quadrant. More formally, an excursion $\Gamma$ of length $\ell$ is a
sequence of points $\gamma = (w_0,\dots,w_{\ell})$ on $(\NN_{\geq 0})^2$ such
that $w_j = (j,y_j)$, with $y_0=y_\ell =0$, and if $w_i = (i,0)$ for some $i$ then $w_{i+} \neq (i+1,0)$. It is uniquely
encoded by the sequence of its \emph{steps} $e_j := w_{j}-w_{j-1} = 
(1,y_{j+1}-y_j)$. For a step $e=(1,y)$ its degree $\deg(e)$ is equal
to $y$. Steps of degree $0$ are called \emph{horizontal steps}. 

For a given excursion $\Gamma = (w_0,\dots,w_{\ell})$, the set of
points $\SSS(\Gamma):=\{w_1,w_2,\dots,w_\ell\}$ (not counting the origin $w_0=(0, 0)$) 
naturally decomposes as
\[ \SSS(\Gamma) = \bigcup_{n \in \ZZ}\SSS_n(\Gamma), \]
where $\SSS_n(\Gamma)$ is a set of points preceded by a step $(1, n)$. 
We also denote by $\SSS^i(\Gamma) \subset
\SSS(\Gamma)$ the
subset of points with second coordinate equal to $i$, i.e.
\[ \SSS^i(\Gamma) := \{w_j = (j,y_j)\colon y_j=i\}.\]
Additionally, we denote $\SSS_{0}(\Gamma)$ by
$\SSS_{\tiny\rightarrow}(\Gamma)$ to remind that these points are preceded by horizontal steps, and we define $\SSS^i_{\tiny\rightarrow}(\Gamma) := \SSS_{\tiny\rightarrow}(\Gamma) \cap \SSS^i(\Gamma)$.

\smallskip

For an ordered tuple $\vec{\Gamma} = (\Gamma_1,\dots,\Gamma_k)$ of $k$ excursions 
$\Gamma_i$, we will treat $\vec{\Gamma}$ itself as an excursion
obtained by concatenating $\Gamma_1,\dots,\Gamma_k$,
and we define $\SSS_n(\vec{\Gamma}), \SSS^i(\vec{\Gamma}),\SSS_{\tiny\rightarrow}^i(\vec{\Gamma})$, in the same way as before.
We say that \emph{$p = (w_i,w_j)$ is a pairing of degree $n > 0$} if $w_i \in \SSS_{-n}(\vec{\Gamma}), w_j\in\SSS_n(\vec{\Gamma})$, and $w_i$ appears
before $w_j$ in $\vec{\Gamma}$, i.e.~$i<j$.

\smallskip

By definition, a \emph{ribbon path on $k$ sites of lengths $\ell_1,\dots,\ell_k$} 
is a pair $\ribbon = (\vec{\Gamma},\,\mathbf{P}(\ribbon))$ consisting of an ordered tuple 
$\vec{\Gamma}$ of $k$ excursions $\Gamma_1,\dots,\Gamma_k$ of lengths $\ell_1, \dots, \ell_k$, 
respectively, and a set $\mathbf{P}(\ribbon)$ of disjoint
pairings $p_1, \dots, p_q$ on $\vec{\Gamma}$. This notion of ribbon paths was introduced by Moll in~\cite{Moll2023}. We denote by $\mathbf{P}_n(\ribbon)\subset\mathbf{P}(\ribbon)$
the subset of pairings of degree $n$, and define
$\SSS_n(\ribbon) := \SSS_n(\vec{\Gamma}) \setminus \mathbf{P}_{|n|}(\ribbon)$ 
as the set obtained by removing the points belonging to $\mathbf{P}_n(\ribbon)$ from 
$\SSS_n(\vec{\gamma})$. Also, let $\SSS_{\tiny\rightarrow}^i(\ribbon) := \SSS_{\tiny\rightarrow}^i(\vec{\Gamma})$.
Then we have the decomposition\footnote{In this paragraph, we abused the notation: the
  set of pairings $\mathbf{P}_n(\ribbon)$ contains pairs of
  distinct points $(w_i,w_j)$, but we implicitly treated such pairs as the 2-element sets $\{w_i, w_j\}$, for simplicity of notation.
}:
\begin{equation*}\label{decompos_S}
  \SSS(\ribbon) = \bigcup_{i=0}^\infty \SSS^i_{\tiny\rightarrow}(\ribbon)\cup\bigcup_{n=1}^\infty \bigl( \mathbf{P}_n(\ribbon)\cup \SSS_{-n} (\ribbon)\cup \SSS_{n}(\ribbon) \bigr).
\end{equation*}
We will denote $\mathbf{R}(\ell_1,\dots,\ell_k)$ the set of 
ribbon paths on $k$ sites of lengths $\ell_1,\dots,\ell_k$.

The following theorem is a direct combinatorial interpretation of the
operator \sloppy $P L^{\ell_1-2} P^\dagger\cdots PL^{\ell_k-2}P^\dagger$ in terms of ribbon paths, and we leave
its proof as a simple exercise (the full proof of its variant can be
found in~\cite{Moll2023}).

\begin{thm}
  \label{theo:N-SCombi}
  The following identity holds true:
  \begin{multline*}
    P L^{\ell_1-2} P^\dagger\cdots PL^{\ell_k-2}P^\dagger =\\
    \sqrt{\alpha}^{-(\ell_1+\cdots+\ell_k)}\sum_{\substack{\ribbon\in\mathbf{R}(\ell_1,\dots,\ell_k),\\
        |\SSS^0(\ribbon)| =
        k}}\sqrt{\alpha}^{|\SSS_{\tiny\rightarrow}(\ribbon)|}\prod_{n=1}^{\infty}
    (\alpha \cdot n)^{|\mathbf{P}_n(\ribbon)|} \prod_{i=1}^{\infty}(-i\cdot \gamma)^{|\SSS^i_{\tiny\rightarrow}(\ribbon)|} \cdot \prod_{j=1}^{\infty}p_{j}^{|\SSS_{j}(\ribbon)|}\cdot \prod_{m=1}^{\infty}p_{-m}^{|\SSS_{-m}(\ribbon)|}.
    \end{multline*}
\end{thm}
Finally, an excursion $\Gamma$ which has only up steps of degree $1$
is called a~\emph{Łukasiewicz path}. Similarly, a ribbon path
$\ribbon$ whose non-paired up steps are only of degree $1$ is called 
a \emph{Łukasiewicz ribbon path},
i.e.~$\SSS_{n}(\ribbon) = \emptyset,\, \forall\, n\ge 2$, if $\ribbon$ 
is a Łukasiewicz ribbon path. 
We will denote $\mathbf{L}(\ell_1,\dots,\ell_k)$ the set of 
Łukasiewicz ribbon paths on $k$ sites of lengths $\ell_1,\dots,\ell_k$.

Łukasiewicz paths are classical
objects in combinatorics\footnote{See~\cite{FlajoletSedgewick2009} for
an explanation of their name and more background.}. We show here that Łukasiewicz ribbon paths naturally arise in studying polynomials $X_{\ell_1,\dots,\ell_k}(x_0,x_1,\dots)$.

\begin{thm}
  \label{theo:XtoCh}
  The following identities hold true:
  \begin{align}
    B^{(\alpha)}_{\ell_1}\cdots B^{(\alpha)}_{\ell_k} &=
    \sum_{\substack{\ribbon \in
      \mathbf{L}(\ell_1,\dots,\ell_k),\\ |\SSS^0(\ribbon)| = k}}\prod_{n=1}^{\infty} n^{|\mathbf{P}_n(\ribbon)|} \prod_{i=1}^{\infty}(-i\cdot \gamma)^{|\SSS^i_{\tiny\rightarrow}(\ribbon)|} \cdot 
                                                        \Ch_{\mu(\ribbon)}^{(\alpha)}, \label{eq:BooleanCh}\\     
    M^{(\alpha)}_{\ell_1}\cdots M^{(\alpha)}_{\ell_k} &=
                                                        \sum_{\ribbon \in
      \mathbf{L}(\ell_1,\dots,\ell_k)}\prod_{n=1}^{\infty} n^{|\mathbf{P}_n(\ribbon)|} \prod_{i=1}^{\infty}(-i\cdot \gamma)^{|\SSS^i_{\tiny\rightarrow}(\ribbon)|} \cdot 
                                                        \Ch_{\mu(\ribbon)}^{(\alpha)},     \label{eq:MomentCh}
  \end{align}
  where $\mu(\ribbon)$ is the partition given by $(1^{|\SSS_{-1}(\ribbon)|},2^{|\SSS_{-2}(\ribbon)|}, \dots)$.
  \end{thm}

  \begin{proof}
    We start by explaining that equation \eqref{eq:MomentCh} is a direct consequence of \eqref{eq:BooleanCh} and relation \eqref{eq:Moment-Boolean}. Indeed, take Łukasiewicz ribbon path $\ribbon \in
    \mathbf{L}(\ell_1,\dots,\ell_k)$ and consider its points touching the $x$-axis $\SSS^0(\ribbon)$. They must be of the form
    \[\SSS^0(\ribbon) = \{\ell_1^1,\dots, \ell_1^1+\cdots+\ell_1^{n_1},\dots, \ell_1^1+\cdots+\ell_{k-1}^{n_{k-1}}+\ell_k^1,\dots, \ell_1^1+\cdots+\ell_{k}^{n_{k}}\},\]
    where $\sum_{j=1}^{n_i}\ell_i^j = \ell_i$ for each $i = 1,\dots, k$. In particular, we can consider $\ribbon$ as an element of $\mathbf{L}(\ell_1^1,\dots,\ell_1^{n_1},\dots, \ell_k^1,\dots,\ell_k^{n_k})$. Using this decomposition we can rewrite the RHS of~\eqref{eq:MomentCh} as
      \begin{equation*}
        \sum_{n_1,\dots,n_k \geq 1}\sum_{\substack{ \ell_1^1,\dots,\ell_1^{n_1} \geq 1,\\ \ell_1^1+\cdots+\ell_1^{n_1}= \ell_1}}\cdots \sum_{\substack{ \ell_k^1,\dots,\ell_k^{n_k} \geq 1,\\ \ell_k^1+\cdots+\ell_k^{n_k}= \ell_k}}\sum_{\substack{\ribbon \in
      \mathbf{L}(\ell_1^1,\dots,\ell_1^{n_1},\dots, \ell_k^1,\dots,\ell_k^{n_k}),\\ |\SSS^0(\ribbon)| = n_1+\cdots +n_k}}\prod_{n=1}^{\infty} n^{|\mathbf{P}_n(\ribbon)|} \prod_{i=1}^{\infty}(-i\cdot \gamma)^{|\SSS^i_{\tiny\rightarrow}(\ribbon)|} \cdot 
                                                        \Ch_{\mu(\ribbon)}^{(\alpha)},
                                                      \end{equation*}
                                                      which, by \eqref{eq:BooleanCh}, is equal to
                                                            \begin{equation*}
        \sum_{n_1,\dots,n_k \geq 1}\sum_{\substack{ \ell_1^1,\dots,\ell_1^{n_1} \geq 1,\\ \ell_1^1+\cdots+\ell_1^{n_1}= \ell_1}}\cdots \sum_{\substack{ \ell_k^1,\dots,\ell_k^{n_k} \geq 1,\\ \ell_k^1+\cdots+\ell_k^{n_k}= \ell_k}}\sum_{\substack{\ribbon \in
      \mathbf{L}(\ell_1^1,\dots,\ell_1^{n_1},\dots, \ell_k^1,\dots,\ell_k^{n_k}),\\ |\SSS^0(\ribbon)| = n_1+\cdots +n_k}}B^{(\alpha)}_{\ell^1_1}\cdots B^{(\alpha)}_{\ell^{n_1}_1}\cdots B^{(\alpha)}_{\ell^1_k}\cdots B^{(\alpha)}_{\ell^{n_k}_k}.
\end{equation*}
Relation \eqref{eq:Moment-Boolean} finishes the proof of~\eqref{eq:MomentCh}.

We now prove~\eqref{eq:BooleanCh}.
Using the fact that $J_\lambda^{(\alpha)}\big|_{p_i =
  \delta_{i,1}} = 1$ (see \cref{thm Jack formula}), we can use
\cref{theo:Nazarov-Sklyanin} to write $B^{(\alpha)}_{\ell_1}\cdots
B^{(\alpha)}_{\ell_k}$ as $\bigg(P L^{\ell_1-2} P^\dagger\cdots PL^{\ell_k-2}P^\dagger J_\lambda^{(\alpha)}\bigg)\bigg|_{p_i =
        \delta_{i,1}}$. \cref{theo:N-SCombi} allows to further rewrite

      \begin{align}
B^{(\alpha)}_{\ell_1}\cdots
B^{(\alpha)}_{\ell_k} &= \sqrt{\alpha}^{-(\ell_1+\cdots+\ell_k)}\sum_{\substack{\ribbon\in\mathbf{R}(\ell_1,\dots,\ell_k),\\ |\SSS^0(\ribbon)| = k}}\sqrt{\alpha}^{|\SSS_{\tiny\rightarrow}(\ribbon)|}\cdot \nonumber\\
        \cdot &\prod_{n=1}^{\infty} (\alpha\cdot n)^{|\mathbf{P}_n(\ribbon)|}\prod_{i=1}^{\infty}(-i\cdot \gamma)^{|\SSS^i_{\tiny\rightarrow}(\ribbon)|} \cdot \bigg(\prod_{j=1}^{\infty}p_{j}^{|\SSS_{j}(\ribbon)|}\cdot \prod_{m=1}^{\infty}p_{-m}^{|\SSS_{-m}(\ribbon)|}J_\lambda^{(\alpha)}\bigg)\bigg|_{p_i =
        \delta_{i,1}} .       \label{eq:MinL}
    \end{align}
    Fix a ribbon path $\ribbon \in \mathbf{R}(\ell_1,\dots,\ell_k)$ and note that
    \begin{equation}\label{eq:pomoc}
      \bigg(\prod_{j=1}^{\infty}p_{j}^{|\SSS_{j}(\ribbon)|}\cdot \prod_{m=1}^{\infty}p_{-m}^{|\SSS_{-m}(\ribbon)|}J_\lambda^{(\alpha)}\bigg)\bigg|_{p_i=
        \delta_{i,1}} = 0
      \end{equation}
    whenever there exists $j > 1$ such that $\SSS_{j}(\ribbon) \neq \emptyset$. In other terms, if $\ribbon$ is not a Łukasiewicz ribbon path, then its contribution into \eqref{eq:MinL} is zero, and we can replace the set $\mathbf{R}(\ell_1,\dots,\ell_k)$ in~\eqref{eq:MinL} by $\mathbf{L}(\ell_1,\dots,\ell_k)$. For a Łukasiewicz ribbon path \eqref{eq:pomoc} simplifies to:
        \begin{align}\label{eq:pomoc2}
      \bigg(p_{1}^{|\SSS_{1}(\ribbon)|}\cdot \prod_{m=1}^{\infty}p_{-m}^{|\SSS_{-m}(\ribbon)|}&J_\lambda^{(\alpha)}\bigg)\bigg|_{p_i=
      \delta_{i,1}} = \alpha^{\ell(\mu(\ribbon))}\bigg(p_{1}^{|\SSS_{1}(\ribbon)|}\cdot
                                                                                                \prod_{j=1}^{\infty}\left(
                                                                                                \frac{j
                                                                                                \partial}{\partial
                                                                                                p_j}\right)^{m_j(\mu(\ribbon))}J_\lambda^{(\alpha)}\bigg)\bigg|_{p_i=
      \delta_{i,1}} \nonumber \\
      &= \alpha^{\ell(\mu(\ribbon))}\binom{|\lambda|-|\mu(\ribbon)|+|\SSS_{-1}(\ribbon)|}{|\SSS_{-1}(\ribbon)|}\cdot z_{\mu(\ribbon)}\left[ p_{1^{|\lambda|-|\mu(\ribbon)|} \cup \mu(\ribbon)}\right]
      J_\lambda^{(\alpha)}\nonumber\\
      &= \sqrt{\alpha}^{|\mu(\ribbon)|+\ell(\mu(\ribbon))}\Ch^{(\alpha)}_{\mu(\ribbon)}.
    \end{align}
    Finally, notice that for any ribbon path $\ribbon$ of length $\ell$ one has
    \[ \sum_{n=1}^\infty n\big(|\SSS_{-n}(\ribbon)|-|\SSS_{n}(\ribbon)|\big)=0, \ \ \ \ \sum_{n=1}^\infty\big(|\SSS_{-n}(\ribbon)|+|\SSS_{n}(\ribbon)|\big) = \ell - |\SSS_{\tiny\rightarrow}(\ribbon)|-2|\mathbf{P}(\ribbon)|,\]
   thus for any Łukasiewicz ribbon path $\ribbon\in\mathbf{L}(\ell_1,\dots,\ell_k)$
   \begin{align}\label{eq:RibbonProp}
     |\mu(\ribbon)|+\ell(\mu(\ribbon)) = \sum_{n \geq
     1}(n+1)|\SSS_{-n}(\ribbon)| &= \sum_{n \geq
                                   1}\big(|\SSS_{-n}(\ribbon)|+|\SSS_{n}(\ribbon)|\big)
                                   = \nonumber\\ &=\ell_1+\cdots+\ell_k -
                                          |\SSS_{\tiny\rightarrow}(\ribbon)|-2|\mathbf{P}(\ribbon)|.
                                          \end{align}
    In particular the RHS of \eqref{eq:pomoc2} is equal to $\sqrt{\alpha}^{\ell_1+\cdots+\ell_k - |\SSS_{\tiny\rightarrow}(\ribbon)|-2|\mathbf{P}(\ribbon)|}\Ch^{(\alpha)}_{\mu(\ribbon)}$ for a Łukasiewicz ribbon path $\ribbon \in \mathbf{L}(\ell_1,\dots,\ell_k)$. Plugging this into \eqref{eq:MinL} yields precisely the desired identity~\eqref{eq:BooleanCh}, which finishes the proof.
  \end{proof}

  \subsubsection{Consequences}

  Before we conclude integrality of Lassalle's conjecture let us point several applications of \cref{theo:XtoCh}.
  
  In the special case $\alpha=1$ a problem of positivity between
  Boolean cumulants and normalized characters of the symmetric group
  was posed by Rattan in Śniady in~\cite{RattanSniady2008}, and has been proven very recently by Koshida~\cite{Koshida2021} by use of Khovanov's Heisenberg category. 
  Koshida, however, was not able to find an explicit interpretation of
  the positivity so that he did not provide the $\alpha=1$ case of our
  formula \cref{eq:BooleanCh} (his proof was a complicated induction)
  and left it as an open problem. \cref{theo:XtoCh} solves this
  problem and provides an explicit combinatorial interpretation for general $\gamma$; in particular in the special case $\alpha=1$ it gives an alternative proof to the work of Koshida.

  Furthermore, it implies the following theorem.

      \begin{thm}
    \label{theo:mainInt}
    The following $\ZZ[\gamma]$-algebras are all equal:
    \[ \ZZ[\gamma][\Ch^{(\alpha)}_\mu\colon \mu \in \YY] = \ZZ[\gamma, M_2^{(\alpha)},
    M_3^{(\alpha)},\dots] = \ZZ[\gamma, R_2^{(\alpha)},
    R_3^{(\alpha)},\dots] = \ZZ[\gamma, B_2^{(\alpha)},
    B_3^{(\alpha)},\dots].\]
\end{thm}

This theorem gives the biggest progress so far towards another conjecture of Lassalle from~\cite{Lassalle2009} that postulates that $\Ch^{(\alpha)}_\mu$ is a polynomial in $\gamma, R_2^{(\alpha)},
    R_3^{(\alpha)},\dots$ with positive integer coefficients (this
    formulation, which is a more precise version of the original
    Lassalle's conjecture was formulated in~\cite[Section 2.5]{DolegaFeray2016}). Rationality of the coefficients of this polynomial (called \emph{Kerov polynomial} for Jack characters) was proven in~\cite{DolegaFeray2016}, and the top-degree part of the Kerov polynomial was found by Śniady in~\cite{Sniady2019}; these properties have found applications for studying random Young diagrams~\cite{DolegaFeray2016,DolegaSniady2019,CuencaDolegaMoll2023}.

  \begin{proof}[Proof of \cref{theo:mainInt}]
    
    The last two equalities are well-known to the experts and follow from \cref{prop:Mom-Bool-Free} and the fact that the equations \eqref{eq:Moment-Boolean} and \eqref{eq:Moment-Free} are invertible over $\ZZ$.

    In order to prove the first equality, consider the lexicographic order on the set of partitions $\mu$ of size $n$ and extend it to the set of partitions of size at most $n$ by defining $\nu \leq \mu$ if $|\nu| < |\mu|$, or $|\nu| = |\mu|$ and $\nu \leq_{\lex} \mu$. Notice then the following fact:
    for any partition $\mu$ of length $\ell$ we have
    \[ M^{(\alpha)}_{\mu_1+1}\cdots M^{(\alpha)}_{\mu_\ell+1} =
      \Ch^{(\alpha)}_{\mu}+\sum_{\rho < \mu}a_\rho^\mu(\gamma) \Ch^{(\alpha)}_{\rho},\]
    where $a_\rho^\mu(\gamma) \in \ZZ[\gamma]$.
    Indeed, \eqref{eq:MomentCh} implies that the contribution of
    $\Ch^{(\alpha)}_{\rho}$ comes from $\ribbon \in
    \mathbf{L}(\mu_1+1,\dots,\mu_{\ell(\mu)}+1)$ with $|\mu(\ribbon)|
    = \rho$. Therefore \cref{eq:RibbonProp} implies that
\[ |\mu(\ribbon)| = |\mu|+\ell(\mu)-
                                         \sum_{n \geq
                                            1}|\SSS_{n}(\ribbon)|-|\SSS_{\tiny\rightarrow}(\ribbon)|-2|\mathbf{P}(\ribbon)|
                                          \leq |\mu|,\]
                                        where the last inequality is
                                        an equality
                                        if and only if
                                        $|\SSS_{\tiny\rightarrow}(\ribbon)|=2|\mathbf{P}(\ribbon)|$
                                        and $\sum_{n \geq
                                            1}|\SSS_{n}(\ribbon)| =
                                          \ell(\mu)$. There is a
                                          unique $\ribbon \in
                                          \mathbf{L}(\mu_1+1,\dots,\mu_{\ell(\mu)}+1)$
                                          that satisfies these
                                          conditions: $\Gamma =
                                          (\Gamma_1,\dots,\Gamma_{\ell(\mu)})$,
                                          where $\Gamma_i$ is given by
                                          an up-step of degree $\mu_i$
                                          followed by $\mu_i$
                                          down-steps. In particular
                                          $\mu(\ribbon) = \mu$, which
                                          implies that the matrix $(a_\rho^\mu(\gamma))_{\rho,\mu}$ is
    uni-triangular over $\ZZ[\gamma]$, thus it is invertible. This
    finishes the proof.  
  \end{proof}

 We have the following corollary.

  \begin{cor}
    The normalized Jack characters  expressed in the Stanley coordinates $(-1)^{|\mu|}z_\mu\tJch_\mu(\bfs,\bfr)$ are polynomials in the variables $b,-s_1,-s_2,\dots,r_1,r_2,\dots$ with integer coefficients, where $b := \alpha-1$.
  \end{cor}

  \begin{proof}
    Strictly from the definition~\eqref{eq:Cauchy-St} of the moments
    one has
    \[ \sqrt{\alpha}^{\ell}\cdot M^{(\alpha)}_\ell(\bfs^\bfr) =
      [z^{-\ell-1}]\prod_{i=1}^{k}\frac{1}{z+r_1+\cdots+r_k}\prod_{i=1}^k\frac{z-\big(\alpha\cdot
        s_i-(r_1+\cdots+r_{i-1})\big)}{z-\big(\alpha\cdot
        s_i-(r_1+\cdots+r_{i})\big)}, \]
    where we extract the coefficient of $z^{-\ell-1}$ in the above
    rational function treated as a formal power series in $z^{-1}$. In
    particular it is clear that $\sqrt{\alpha}^{\ell} M^{(\alpha)}_\ell $ is a polynomial in $b,-s_1,-s_2,\dots,r_1,r_2,\dots$ with integer coefficients. Then the result follows from the definition of $\gamma := b \cdot \sqrt{\alpha}^{-1}$, the definition of the normalized Jack characters~\eqref{eq:JackCharacter}, and \cref{theo:mainInt}.
    \end{proof}

\section{The combinatorial model and differential equations}\label{sec comb model}

\subsection{Statistics of non-orientability}\label{ssec SON}
The purpose of this section is to define a family of statistic of non-orientability (see \cref{def SON}) on layered maps. We start by some general definitions related to non-orientable maps.

Let $M$ be a bipartite map and let $c_1$ and $c_2$ be two corners of
$M$ of different colors. Then we have two ways to add an edge to $M$
between these two corners (see \cref{fig edge-types}). We denote by $e_1$ and $e_2$ these edges. We say that the pair $(e_1,e_2)$ is a \textit{pair of twisted edges} on the map $M$ and we say that $e_2$ is obtained by twisting $e_1$. Note that if $M$ is connected and orientable, then exactly one of the maps $M\cup\{e_1\}$ and $M\cup\{e_2\}$ is orientable. 
For a given map with a distinguished edge $(M,e)$, we denote $(\widetilde{M},\tilde{e})$ the map obtained by twisting the edge $e$.  

We recall that $b$ is the parameter related to the Jack parameter $\alpha$ by $b=\alpha-1$.
We now give the definition of a measure of non-orientability due to La Croix.

\begin{defi}\cite[Definition 4.1]{LaCroix2009}
We call a measure of non-orientability \textup{(MON)} a function $\rho$ defined on the set of connected maps $(M,e)$ with a distinguished edge, with values in  $\left\{1,b\right\}$, satisfying the following conditions:
\begin{itemize}
    \item if $e$ connects two corners of the same face of
      $M\backslash\{e\}$, and the number of the faces increases by 1
      by adding the edge $e$ on the map $M\backslash\{e\}$, then
      $\rho(M,e)=1$. In this case we say that $e$ is a
      \textit{border}. 
    \item  if $e$ connects two corners in the same face $M\backslash\{e\}$, and the number of the faces of $M\backslash\{e\}$ is equal to the number of faces of $M$, then $\rho(M,e)=b$. In this case we say that $e$ is a \textit{twist}.
    \item if $e$ connects two corners of two different faces lying in the same connected component of $M\backslash\{e\}$, then $\rho$  satisfies $\rho(M,e)+\rho(\widetilde{M},\tilde{e})=1+b$. In this case we say that $e$ is a \textit{handle}. Moreover, if $M$ is orientable then $\rho(M,e)=1$. 
    \item if $e$  connects two faces lying in two different connected components, then $\rho(M,e)=1$. In this case, we say that $e$ is a \textit{bridge}.
\end{itemize}

Let $M$ be a bipartite map and let $e_1,e_2,\dots,e_d$ be $d$ distinct edges of $M$. For $1\leq i\leq d$, we denote $M_j$ be the map obtained by deleting the edges $e_1$, $e_{2}$,..., $e_{j}$ from $M$.
We define $\rho(M,e_1,e_2,\dots,e_d)$ as the weight obtained by deleting the edges $e_j$ successively:
$$\rho(M,e_1,e_2,\dots,e_d):=\prod_{1\leq j\leq d}\rho(M_j,e_j).$$
\end{defi}

\begin{figure}[t]
\centering

\begin{subfigure}{0.8\textwidth}
\centering

\begin{subfigure}{.4\textwidth}
  \centering
    \includegraphics[width=0.5\textwidth]{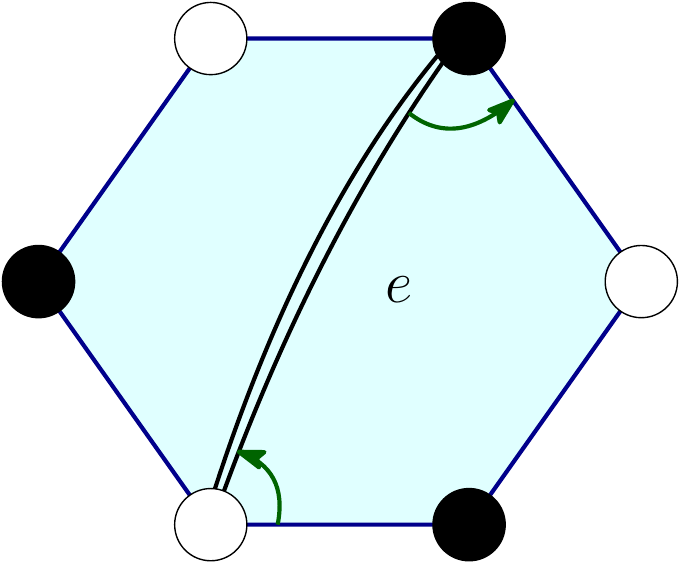}
    \label{border}
\end{subfigure}%
\begin{subfigure}{.4\textwidth}
  \centering
    \includegraphics[width=0.5 \textwidth]{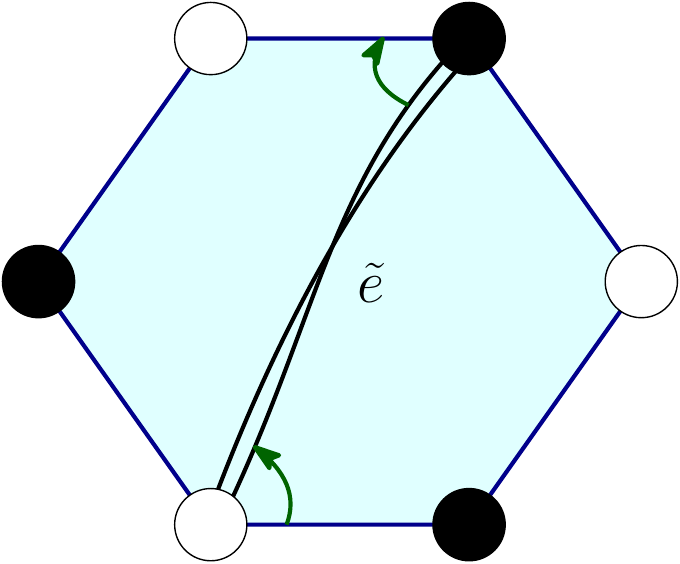}
\label{twist}
\end{subfigure}

\label{border twist}
\caption{A pair of twisted edges $(e,\tilde{e})$ between two corners of the same face; $e$ is a border while $\tilde{e}$ is a twist. }
\end{subfigure}

\begin{subfigure}{0.8\textwidth}
\begin{subfigure}{.45\textwidth}
  \centering
    \includegraphics[width=0.9\textwidth]{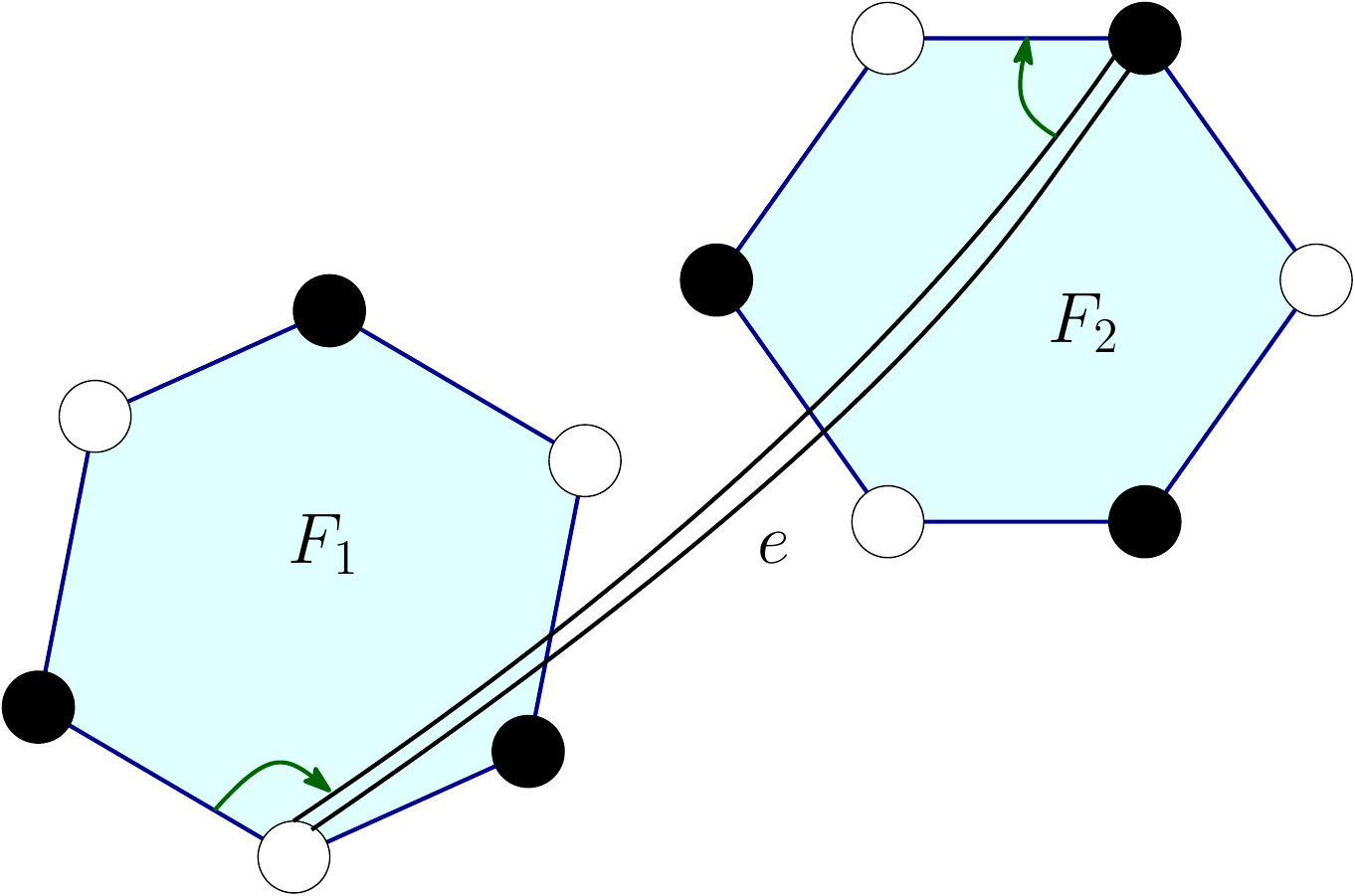}
  \label{handle 1}
\end{subfigure}%
\begin{subfigure}{.45\textwidth}
  \centering
    \includegraphics[width=0.9 \textwidth]{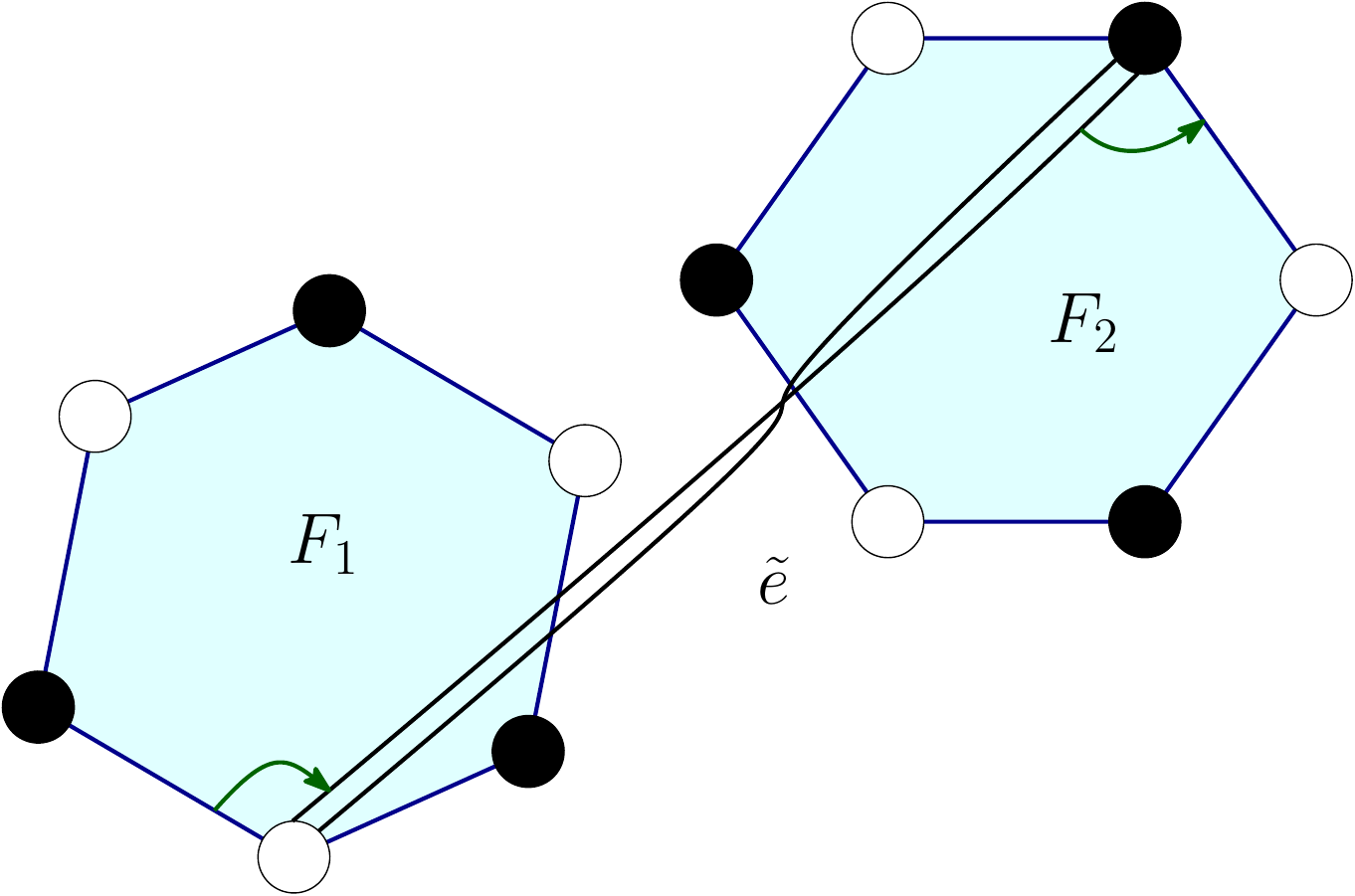}
    \label{handle 2}
\end{subfigure}
 \caption{A pair of twisted edges $(e,\tilde{e})$ between two corners of different faces $F_1$ and $F_2$. If $F_1$ and $F_2$ are on the same connected component then the edges $e$ and $\tilde e$ are handles, otherwise they are bridges. }
\label{handles bridges}
 \end{subfigure}

\caption{The different ways of adding an edge to a map. The added edge is represented each time by a band, that can be twisted at most once. The arrows indicate how the added edge connects the two respective corners.}
\label{fig edge-types}
\end{figure}

Given a MON $\rho$, one can associate a $b$-weight $\rho(M)$ to a layered map $M$ by choosing an order on the edges of $M$.  We start by defining an order on black vertices.

We fix an integer $k\geq1$ and a labelled $k$-layered map $M$. 
To each black vertex $v$ of $M$ we associate the triplet of positive integers $(i,n,j)$, where $i$ is the layer containing the vertex, $n$ is its degree and $j$ is the number given to $v$ by the labelling of the map. By definition, the triplets associated to two distinct black vertices are different. 
We define then a linear order $\prec_M$ on the black vertices of $M$ given by the lexicographic order on the triplets $(-i,n,j)$.  
In other terms, the maximal black vertex with respect to $\prec_M$ is
the vertex contained in the layer of the smallest index, and having maximal degree and maximal label.
Note that when we delete the maximal vertex from a $k$-layered map and all the edges incident to it, the map obtained is also $k$-layered.

\begin{rmq}
Note that if $M$ is a labelled layered map, and $N$ is one of its
connected components then $N$ inherits a structure of labelled layered
map from $M$. Moreover, the order $\prec_N$ is the order induced by $\prec_M$.
\end{rmq}

We now introduce a statistic of non-orientability on layered labelled maps which will have a key role in this paper.
\begin{defi}[Statistic of non-orientability on $k$-layered maps]\label{def maps stat}
Let $\rho$ be a \textup{MON} and let $M$ be a connected labelled $k$-layered map with $n$ edges. We define the $b$-weight $\rho(M)$ of $M$, as the weight obtained by decomposing $M$ recursively  as follows. If $M$ is the empty map then $\rho(M)=1$. Otherwise, let $v$ be the maximal black vertex of $M$  with respect to $\prec_M$, and let $c$ be its root and $d$ its degree. We denote by $e_1,\dots, e_d$ the edges incident to $v$ as they appear when we turn around $v$ starting from the root $c$. Let $M'$ be the map obtained from $M$ by deleting $v$ and all the edges incident to it.
Then, by definition, 
$$\rho(M):=\rho(M,e_1,\dots,e_d)\cdot\rho(M').$$
We extend this definition to disconnected maps by multiplicativity; if $M$ is a labelled $k$-layered map with $m$ connected components $M_1,\dots,M_m$, then 
$\rho(M):=\prod_{1\leq i \leq m}\rho(M_i).$

Finally, we define the statistic $\vartheta_\rho$ on labelled layered
maps with non-negative integer values, defined for every $M$ by $\rho(M)=b^{\vartheta_\rho(M)}$. 
\end{defi}

It follows from the definitions that $\vartheta_\rho$ is a statistic of non-orientability over labelled bipartite maps.
\begin{rmq}\label{rq statistics}
The statistic $\vartheta_\rho$ is a variant of the statistics introduced in \cite{LaCroix2009,Dolega2017a,ChapuyDolega2022}, except that the decomposition algorithm used here is different and depends not only on the structure of the map but also on its labelling.
\end{rmq}

\subsection{Catalytic operators in \texorpdfstring{$Y$}{}}
\label{ssec cat Y}
In this section, we fix an integer $k\geq1$, and $k$ variables $s_1,s_2,\dots,s_k$.
We consider the algebra
$$\mcP:=\Span_{\mathbb{Q}(b)}\{p_\lambda\}_{\lambda\in\mathbb{Y}}.$$
We also consider an alphabet $Y:=\{y_0,y_1,\dots\}$ and the space\footnote{This notation is slightly different from the one used in \cite{ChapuyDolega2022}.} $\PY$ 
$$\PY:=\Span_{\mathbb{Q}(b)}\left\{y_ip _\lambda\right\}_{i\in \mathbb{N},\lambda\in\mathbb{Y}}.$$

We fix a MON $\rho$.
Let $M$ be a $k$-layered map.
We define its \textit{marking} $\kappa(M)$ by
$$\kappa(M):=\frac{\rho(M)}{2^{|\Vbul(M)|-\cc(M)}\alpha^{cc(M)}}p_{\tf(M)}\prod_{1\leq i\leq k}(-\alpha s_{i})^{|\mathcal{V}_\circ^{(i)}(M)|}\in \mcP[s_1,\dots,s_k],$$
where $\tf(M)$ is the face-type of $M$ defined in \cref{ssec maps}.

\begin{defi}
We say that a layered map $M$ is rooted if it has a distinguished oriented black corner $c$ in layer 1, called  the root of the map. A rooted layered map is labelled if all the black vertices are numbered as in \cref{def layered maps} except for the root vertex $v_c$. As in \cref{def maps stat}, we define an order $<_{(M,c)}$ on the black vertices of a labelled rooted layered map $(M,c)$ with the convention that the root vertex is always maximal. Finally we denote $\rho(M,c)$ the $b$-weight associated to $(M,c)$ with respect to $<_{(M,c)}$.
\end{defi} 

We associate to a rooted labelled $k$-layered map$(M,c)$ the marking $\kappa(M,c)$ defined by:
    $$\kappa(M,c):=\frac{\rho(M,c)}{2^{|\Vbul(M)|-\cc(M)}\alpha^{\cc(M)}} y_{\deg(f_c)}\prod_{f\neq f_c} p_{\deg(f)}\prod_{1\leq i\leq k}(-\alpha s_{i})^{|\mathcal{V}_\circ^{(i)}(M)|}\in\PY[s_1,\dots,s_k],$$
where $\deg(f_c)$ is the degree of the root face, and the first product runs over the faces of $M$ different from the root face.

We use the catalytic operators $Y_+, \GY\colon \PY \rightarrow \PY$ introduced in \cite{ChapuyDolega2022}:
\begin{equation}\label{eq Y+}
Y_+=\sum_{i\geq 0}y_{i+1}\frac{\partial}{\partial y_i},
\end{equation}

\begin{equation}\label{eq LambdaY}
\GY=(1+b)\cdot\sum_{i,j\geq1}y_{i+j}\frac{i\partial^2}{\partial p_i\partial y_{j-1}}+\sum_{i,j\geq1} y_{i}p_j\frac{\partial}{\partial y_{i+j-1}} +b\cdot \sum_{i\geq1}y_{i+1}\frac{i\partial}{\partial y_i}.
\end{equation}

Note that the operator $\GY$ has a similar structure to the Laplace--Beltrami
operator $D_\alpha$, but it additionally uses variables from the
family $Y$. This justifies the name ``catalytic'', as these variables
will play the same role as the catalytic variable in the classical
Tutte decomposition, well-known to the combinatorial community. 
The following proposition gives a combinatorial interpretation for the
operator $\GY$, which corresponds to the particular case $k=1$ in
\cite[Proposition 4.4]{ChapuyDolega2022}. We give here the main
arguments of the proof. 

\begin{prop}[\cite{ChapuyDolega2022}]\label{prop Gamma}
Let $(M,c)$ be a rooted $k$-layered map. Then for every MON $\rho$, we have
$$\GY \kappa(M,c)=\sum_{e}\kappa(M\cup\{e\},c),$$
where the sum is taken over all ways to add an edge $e$ connecting $c$
to some white corner $c'$ (without adding a new white vertex).
\end{prop}
\begin{proof}

First, notice that by adding such an edge $e$, the maximal
vertex in $M\cup\{e\}$ is the same as the maximal vertex in $M$, therefore $\rho(M\cup\{e\}) =
\rho(M\cup\{e\},e) \cdot \rho(M)$. We distinguish three cases.
\begin{itemize}
    \item The two corners $c$ and $c'$ lie in distinct faces of respective sizes $i$ and $j$. When we add the edge $e$ we form a face of size $i+j+1$. Let us show that this case corresponds to the first term of the operator $\GY$. Let $\tilde{e}$ denote the edge obtained by twisting $e$, see \cref{ssec SON}. 
    
    If the two faces lie in the same connected component of $M$, then $e$ is a handle and  by definition of a MON  $$\rho(M\cup\{e\},e)+\rho(M\cup\{\tilde e\},\tilde e)=1+b,$$
    and this explains the factor $1+b$ in the first terms of $\Gamma_Y$.
    
    If the two faces lie in two different connected components, then $e$ is a bridge and $$\rho(M\cup\{e\},e)=\rho(M\cup\{\tilde e\},\tilde e).$$ In this case, the factor $1+b=\alpha$ in the first term of $\Gamma_Y$ is related to the fact that the number of connected components of $M$ decreases by 1.
    
    \item The two corners $c$ and $c'$ lie in the same face of degree $i+j-1$ and the added edge $e$ is a border, which splits the face into two faces of respective degrees $i$ and $j$. Then $\rho(M\cup\{e\},e)=1$.
    \item The two corners $c$ and $c'$ lie in the same face of degree $i$ and the added edge $e$ is a twist. By adding $e$ we form a face of degree $i+1$ and $\rho(M\cup\{e\},e)=b$.\qedhere
\end{itemize}
\end{proof}

\subsection{Proof of \texorpdfstring{\cref{prop expr F}}{}}\label{ssec proof expr F}
We consider the operator  $\Theta_Y\colon\PY \rightarrow \mcP$, defined by
$$\Theta_Y:=\sum_{i\geq1}p_i\frac{\partial}{\partial y_i}.$$
Applied on the marking of a rooted map, $\Theta_Y$ allows to forget
the root and to obtain the marking of an unrooted map. In other words,
if $(M,c)$ is a rooted map with vertex root $v_c$, and $\deg(v_c)$ is
maximal among black vertices in the layer $1$, then 
$$\Theta_Y\cdot\kappa(M,c)=\kappa(M),$$
where $M$ is the map obtained by numbering the root vertex such that
it remains maximal (which is possible due to the definition of the
order $\prec_M$, and our assumption on $\deg(v_c)$) and then
forgetting the map root $c$ (but the vertex $v_c$ remains rooted in
$c$, where this rooting is the one which is required in the definition of
a layered map). 
For $n\geq0$, and variable $u$, the operator $\B_n\colon \mcP \rightarrow \mcP[u]$ is defined by
\begin{equation*}\label{eq def Bn}
  \B_n(\bfp,u):=\Theta_Y\left(\GY+u Y_+\right)^n\frac{y_0}{1+b}.  
\end{equation*}


\begin{prop}\label{prop Bn}
Let $M$ be a $k$-layered bipartite map and let $n\geq \left(\nu_\bullet^{(1)}(M)\right)_1$. Then 
$$\B_n(\bfp,-\alpha s_1)\cdot \kappa(M)=\sum_{M'}\kappa(M'),$$
where the sum is taken over all $k$-layered maps obtained by adding a
black vertex of degree $n$ and label
$m_n\left(\nu_\bullet^{(1)}(M)\right)+1$ to the map $M$ in the layer 1 (using possibly new white
vertices which are necessarily in the layer 1).
\end{prop}

\begin{proof}
We start by adding an isolated black vertex $v_c$ of degree 0 to $M$ in layer $1$ --- this corresponds to the multiplication by $y_0/\alpha$ (the division by $\alpha$ is due to the fact that we increase the number of connected components of the map by 1). We have two ways to add an edge incident to $v$:
\begin{itemize}
\item We add an edge to an existing black corner in the map, this corresponds to the term $\GY$ (see \cref{prop Gamma}).  
\item We add an edge connected to a new white vertex. Since the weight of a white vertex in the layer $1$ is $-\alpha s_1$, this corresponds to $(-\alpha s_1)Y_+$. 
\end{itemize}
We apply $\Theta_Y$ to forget the root of the map $M'$ as explained above, in particular, we number the root vertex $v_c$ by $m_n\left(\nu_\bullet^{(1)}(M)\right)+1$.
The algorithm ensures that $\rho(M')$ is computed by erasing the edges in the same order as before, and the proof is finished.
\end{proof}

For each $\ell,k\geq0$, we define $\C_{\ell,k}(\bfp)\colon \mcP
\rightarrow \mcP$ by
$$\C_{\ell,k}(\bfp):=[u^\ell]\B_{k+\ell}(\bfp,u).$$
From \cref{prop Bn} we get that $\C_{\ell,k}$ acts on the marking of a
bipartite map by adding a black vertex of degree $\ell+k$ with $\ell$
new white neighbors. The operators $\C_\ell \colon \mcP \rightarrow \mcP \llbracket t \rrbracket_+$ are then defined for $\ell\geq0$ as the marginal sums
\begin{equation}
\label{eq:DefCl}
\C_\ell(t,\bfp):=\sum_{k\geq1}\frac{t^{\ell+k}}{\ell+k}\C_{\ell,k}({\bfp})+\mathbbm{1}_{\ell>0}\frac{t^\ell}{\ell}\C_{\ell,0}(\bfp),
\end{equation}
where $\mcP \llbracket t \rrbracket_+$ is the ideal in $\mcP
\llbracket t \rrbracket$ generated by $t$.

Finally, we set  
\begin{equation}
  \label{eq:DegBInfty}
  \Binf(t,\bfp,u)=\sum_{n\geq 1}\frac{t^n}{n}
\B_n(\bfp,u)=\sum_{\ell\geq0}u^\ell \C_\ell(t,\bfp) \colon \mcP
\rightarrow \mcP[u]\llbracket t \rrbracket_+.
\end{equation}
We define the generating series of $k$-layered maps by:
\begin{equation}\label{eq def F}
F^{(k)}\left(t,\bfp,s_1,\dots,s_{k}\right):=\sum_{M\in\Mk}(-t)^{|M|}p_{\tf(M)}\frac{b^{\vartheta_\rho(M)}}{2^{|\Vbul(M)|-\cc(M)}\alpha^{cc(M)}}\prod_{1\leq i\leq k}\frac{(-\alpha s_{i})^{|\mathcal{V}_\circ^{(i)}(M)|}}{z_{\nu_\bullet^{(i)}(M)}}.
\end{equation}
We show that the operator $\Binf$ can be used to build the generating series $F^{(k)}$ of the $k$-layered maps.

\begin{prop}\label{prop expr F}
The functions $F^{(k)}$ satisfy the following induction: $F^{(0)}=1$ and for every $k\geq1$ 
\begin{equation}\label{eq F}
F^{(k)}\left(t,\bfp,s_1,\dots,s_{k}\right)=\exp\left(\Binf(-t,\bfp,-\alpha s_{1})\right)\cdot F^{(k-1)}\left(t,\bfp,s_2,\dots,s_{k}\right).
\end{equation}
\end{prop}

In order to prove \cref{prop expr F}, we need the following
commutation relation satisfied by the operators $\B_\ell$ (the proof is postponed to \cref{sec vanishing}).
\begin{prop}\label{prop comm B}
Let $m,\ell\geq1$ and let $u$ be a formal parameter. Then, 
$$\left[\B_\ell(\bfp,u),\B_m(\bfp,u)\right]=0.$$
\end{prop}

\begin{proof}[Proof of \cref{prop expr F}]
For $k=0$, we know that $F^{(0)}=1$ since the only map with 0 layers is the empty map.
Fix now $k\geq 1$. From the definitions 
$$F^{(k-1)}(t,\bfp,s_2\dots,s_k)=\sum_{M\in \mathcal{M}_{k-1}}(-t)^{|M|}p_{\tf(M)}\frac{\rho(M)}{2^{|\Vbul(M)|-\cc(M)}\alpha^{\cc(M)}}\prod_{1\leq i\leq k-1}\frac{(-\alpha s_{i+1})^{|\Vcirc^{(i)}(M)|}}{z_{\nu_\bullet^{(i)}(M)}}.$$
We can rewrite this sum over $k$-layered maps with empty layer 1, by reindexing the layer $j$ by $j+1$ for $1\leq j\leq k-1$. Hence,
$$F^{(k-1)}(t,\bfp,s_2\dots,s_k)=\sum_{M}(-t)^{|M|}\kappa(M)\prod_{1\leq i\leq k-1}\frac{1}{z_{\nu_\bullet^{(i+1)}(M)}},$$
where the sum is taken over all labelled $k$-layered maps with empty layer 1. 
Fix such a map $M$. To obtain a labelled $k$-layered map $M'$ from $M$, we should add the layer 1 (possibly empty). We proceed as follows.
\begin{itemize}
    \item We start by fixing a non-negative integer $d$ and a partition $\mu$ of length $d$ (this partition is empty if $d=0$).
    \item We add successively for each $1\leq i\leq d$ a black vertex $v_i$ of degree $\mu_i$, using possibly new white vertices, such that all the added vertices are in the layer $1$ of the map.
    \item The edges $e_1,\dots, e_{\mu_i}$ incident to a vertex $v_i$ are added successively in a cyclic order around the vertex. The vertex root is chosen such that if we travel around $v_i$ starting from the root corner we see the edges  in the following order $e_{\mu_{i}},\dots,e_1$.
\end{itemize} 
Note that  one has by definition $\nu^{(1)}_\bullet(M')=\mu$ and $\nu^{(j)}_\bullet(M')=\nu^{(j)}_\bullet(M)$ for $2\leq j\leq k$.
\cref{prop Bn} implies that the generating series of $k$-layered maps $M'$ which are obtained from $M$ as described above can be expressed as follows
\begin{multline*}
  \sum_{M'}(-t)^{|M'|}\kappa(M')\prod_{1\leq i \leq k}\frac{1}{z_{\nu_\bullet^{(i)}(M')}}=\sum_{\mu\in\mathbb{Y}}\left(\prod_{j \geq
      1}\frac{1}{m_j(\mu)!}\right)\cdot \\ \cdot \frac{(-t)^{\mu_{\ell(\mu)}}\B_{\mu_{\ell(\mu)}}(\bfp,-\alpha
    s_{1})}{\mu_{\ell(\mu)}}\dots\frac{(-t)^{\mu_{1}}\B_{\mu_{1}}(\bfp,-\alpha
    s_{1})}{\mu_1}(-t)^{|M|}\kappa(M)\cdot \prod_{1\leq i \leq k-1}\frac{1}{z_{\nu_\bullet^{(i+1)}(M)}}.
\end{multline*}
Since the operators $\B_{\mu_i}$ commute (see \cref{prop comm B} above), and since there are $\ell(\mu)!\prod_{j\geq 1}\frac{1}{m_j(\mu)!}$ reorderings $\gamma$ of $\mu$, the RHS of the last equation can be rewritten as follows
$$\sum_{\ell \geq 1}\sum_{n_1,\dots,n_\ell \geq 1}\frac{1}{\ell!}\frac{(-t)^{n_\ell}\B_{n_\ell}(\bfp,-\alpha s_{1})}{n_\ell}\cdots\frac{(-t)^{n_1}\B_{n_1}(\bfp,-\alpha s_{1})}{n_1}(-t)^{|M|}\kappa(M) \cdot \prod_{1\leq i \leq k-1}\frac{1}{z_{\nu_\bullet^{(i+1)}(M)}}.$$
Hence 
$$\sum_{M'}(-t)^{|M'|}\kappa(M')\prod_{1\leq i \leq k}\frac{1}{z_{\nu_\bullet^{(i)}(M')}}=\exp\left(\sum_{n\geq 1}(-t)^n\frac{\B_n(\bfp,-\alpha s_{1})}{n}\right)\cdot \kappa(M) \cdot \prod_{1\leq i \leq k-1}\frac{1}{z_{\nu_\bullet^{(i+1)}(M)}},$$
and we deduce that 
$$F^{(k)}(t,\bfp,s_1,\dots, s_k)=\exp\left(\sum_{n\geq 1}(-t)^n\frac{\B_n(\bfp,-\alpha s_{1})}{n}\right)\cdot F^{(k-1)}(t,\bfp,s_2,\dots,s_k).\qedhere$$
\end{proof}


\section{The vanishing property}\label{sec vanishing}
In this section we consider expressions in two different alphabets $\bfp:=(p_1,p_2,\dots,)$ and $\bfq:=(q_1,q_2,\dots)$. We will use repeatedly without further mention the fact that operators and coefficient extraction which depend on different alphabets trivially commute.

Let $\lambda$ be a partition, of length $\ell\geq 1$. Then we define 
\begin{equation}\label{eq Flambda}
    F(\lambda):=F^{(\ell)}(t,\bfp,\lambda_1,\lambda_2,\dots,\lambda_\ell)
\end{equation}
where $F^{(\ell)}$ is the generating series of $k$-layered maps given by \cref{eq def F}. The main purpose of this section is to prove the following theorem.
\begin{thm}[Vanishing property]\label{thm vanishing} 
Let $\lambda$ be a partition of size $n$. Then the function
$F(\lambda)$ is a polynomial in $t$ of degree less or equal than $n$. In other terms, if $m>n$ then
$$[t^m]F(\lambda)=0.$$
\end{thm}

Combinatorially, the vanishing property is equivalent to saying that
the total contribution of maps with more than $|\lambda|$ edges is
zero in the series $F(\lambda)$. In the cases $b=0$ and
$b=1$, a combinatorial proof of this property was given in
\cite{FeraySniady2011a} and \cite{FeraySniady2011},
respectively. However, such a proof does not seem to work for the
general $b$ because of the presence of the $b$-weight.

\subsection{The function \texorpdfstring{$\tau_b$}{}}
We consider the function $\tau_b$ introduced in \cite{ChapuyDolega2022}.
$$\tau_b(t,\bfp,\bfq,\underline u):=\sum_{\xi\in
  \mathbbm{Y}}t^{|\xi|}\frac{\Jxi(\bfp)\Jxi(\bfq)\Jxi(\underline
  u)}{j^{(\alpha)}_\xi} \in \mathbb{Q}(\alpha)[\bfp,\bfq,u] \llbracket t \rrbracket.$$




The following theorem, due to Chapuy and the second author, will play a crucial role in this section.
\begin{thm}\cite[Theorem 5.7]{ChapuyDolega2022}\label{thm ChapuyDolega2022}
For any $m\geq1$, we have 
$$t^m\frac{\B_m(\bfp,u)}{m}\cdot\tau_b(t,\bfp,\bfq, \underline u)=\frac{\partial}{\partial q_m}\tau_b(t,\bfp,\bfq, \underline u).$$
\end{thm}

We start by giving a construction of the function $\tau_b$ using the operators $\B_m(\bfp,u)$. This construction is related to the fact that $\tau_b$ is the generating series of bipartite maps, see \cite{ChapuyDolega2022, BenDali2022a}.

\begin{cor}\label{prop tau}
The function $\tau_b$ has the following expression:
$$\tau_b(t,\bfp,\bfq,\underline{u})=\exp\left(\sum_{m\geq1}\frac{t^mq_m}{m}\B_m(\bfp,u)\right)\cdot
1.$$
\end{cor}
\begin{rmq}
  \label{rmq:WellDef}
  Recall that $\B_m(t,\bfp,u) \colon \mcP
\rightarrow \mcP[u]$, therefore
$\exp\left(\sum_{m\geq1}\frac{t^mq_m}{m}\B_m(\bfp,u)\right)\colon \mcP
\rightarrow \mathbb{Q}(\alpha)[\bfp,\bfq,u] \llbracket t \rrbracket$ is a well-defined operator.
  \end{rmq}
\begin{proof}
Fix an integer $n\geq1$ and a partition $\mu\vdash n$.
We prove that 
\begin{equation}\label{eq prop tau}
  [t^n q_\mu]\tau_b(t,\bfp,\bfq,\underline{u})=\frac{1}{\ell(\mu)!}\sum_{\gamma\models \mu}\frac{\B_{\gamma_{\ell(\mu)}}(\bfp,u)}{\gamma_{\ell(\mu)}}\cdots \frac{\B_{\gamma_1}(\bfp,u)}{\gamma_1} \cdot1,  
\end{equation}
where the sum is taken over all the reorderings $\gamma$ of $\mu$.
We start by noticing that, for any reordering $\gamma$ of $\mu$,
\begin{align*}
    [t^n q_\mu]\tau_b(t,\bfp,\bfq,\underline{u})
    =[t^n 1_{\bfq}]\left(\prod_{j\geq 1}\frac{1}{m_j(\mu)!}\right)\left(\prod_{1\leq i\leq \ell(\mu)}\frac{\partial}{\partial q_{\gamma_i}}\right)\tau_b(t,\bfp,\bfq,\underline{u}),
\end{align*}
where $[1_{\bfq}]$ denotes the extraction of the constant term in the variables $\bfq$. Since there are $\ell(\mu)!\prod_{j\geq 1}\frac{1}{m_j(\mu)!}$ reorderings $\gamma$ of $\mu$, we can rewrite the last equation as follows
\begin{align*}
    [t^n q_\mu]\tau_b(t,\bfp,\bfq,\underline{u})
    =[t^n 1_{\bfq}]\frac{1}{\ell(\mu)!}\sum_{\gamma\models \mu}\frac{\partial}{\partial q_{\gamma_1}}\cdots \frac{\partial}{\partial q_{\gamma_{\ell(\mu)}}}\tau_b(t,\bfp,\bfq,\underline{u}).
\end{align*}

 
Using \cref{thm ChapuyDolega2022} and the fact that the operators $\frac{\partial}{\partial q_{i}}$ commute with the operators $\B_{j}(\bfp,u)$ we obtain
\begin{align*}
    [t^n q_\mu]\tau_b(t,\bfp,\bfq,\underline{u})
    &=[t^n 1_{\bfq}]\frac{1}{\ell(\mu)!}\sum_{\gamma\models \mu}\frac{t^{\gamma_{\ell(\mu)}}\B_{\gamma_{\ell(\mu)}}(\bfp,u)}{\gamma_{\ell(\mu)}} \cdots \frac{t^{\gamma_1}\B_{\gamma_1}(\bfp,u)}{\gamma_1} \tau_b(t,\bfp,\bfq,u)\\
    &=\frac{1}{\ell(\mu)!}\sum_{\gamma\models \mu}\frac{\B_{\gamma_{\ell(\mu)}}(\bfp,u)}{\gamma_{\ell(\mu)}} \cdots \frac{\B_{\gamma_1}(\bfp,u)}{\gamma_1} [t^0 1_{\bfq}]\tau_b(t,\bfp,\bfq,\underline{u})\\
    &=\frac{1}{\ell(\mu)!}\sum_{\gamma\models \mu}\frac{\B_{\gamma_{\ell(\mu)}}(\bfp,u)}{\gamma_{\ell(\mu)}} \cdots \frac{\B_{\gamma_1}(\bfp,u)}{\gamma_1}\cdot 1.
\end{align*}
This concludes the proof of \cref{eq prop tau} and hence the proof of the proposition.
\end{proof}

A second consequence of \cref{thm ChapuyDolega2022} is \cref{prop comm B}. The key idea of the proof is to consider the action of the commutators $\B_\ell(\bfp,u)$ and $\B_m(\bfp,u)$ on the function $\tau_b(\bfp,\bfq,\underline{u})$ and then extract some coefficient.

\begin{proof}[Proof of \cref{prop comm B}]
  \cref{thm ChapuyDolega2022} implies that
  \begin{equation}\label{eq commutation B tau}
    \left[\B_\ell(\bfp,u),\B_m(\bfp,u)\right]\cdot\tau_b(t,\bfp,\bfq,\underline{u})=0,
  \end{equation}
  where
$[\cdot,\cdot]$ denotes the commutator. Indeed,
\begin{align*}
    t^{\ell+m}\B_\ell(\bfp,u)\B_m(\bfp,u)\cdot\tau_b(t,\bfp,\bfq,\underline{u})
    &=\frac{\partial}{\partial q_m}\frac{\partial}{\partial
      q_\ell}\tau_b(t,\bfp,\bfq,\underline{u}) \\
  = \frac{\partial}{\partial q_\ell}\frac{\partial}{\partial
  q_m}\tau_b(t,\bfp,\bfq,\underline{u}) &= t^{\ell+m}\B_m(\bfp,u)\B_\ell(\bfp,u)\cdot\tau_b(t,\bfp,\bfq,\underline{u}),
\end{align*}
where the first and third equalities are consequences of \cref{thm
  ChapuyDolega2022} and the fact that $\frac{\partial}{\partial q_m}$ and $\B_\ell(\bfp,\underline{u})$ commute.
By extracting the coefficient of $\Jxi(\bfq)$ in \cref{eq commutation B tau}:
\begin{align*}
    \left[\B_\ell(\bfp,u),\B_m(\bfp,u)\right]\cdot\frac{\Jxi(\bfp)\Jxi(\underline{u})}{j^{(\alpha)}_\xi}=0.
\end{align*}
This concludes the proof, since Jack polynomials form a basis of
$\mcP$, and $\frac{\Jxi(\underline{u})}{j^{(\alpha)}_\xi} \neq 0$ by
\cref{eq j alpha} and
\cref{thm Jack formula}.
\end{proof}

\subsection{The space \texorpdfstring{$\Ps$}{}}
We define for any integer $s\geq 1$, the space $\Ps\subset \mcP$ by 
$$\Ps:=\Span_{\mathbb{Q}(b)}\left\{\Jxi(\bfp) \right\}_{\xi\in \mathbbm{Y},\xi_1\leq s}.$$

In this section we prove some properties of the action of $\B_m$ on $\Ps$. 

\begin{lem}\label{lem Ps}
Let $s$ be a positive integer and let $\xi$ be a partition. 
Then $\Jxi(\underline{-\alpha s})=0$ if and only if $\xi_1>s$. 
\begin{proof}
From \cref{thm Jack formula}, we know that $\Jxi(\underline{-\alpha s})=0$ if and only if $\xi$ contains the box $\Box_0=(s+1,1)$ , the only box satisfying $c_\alpha(\Box_0)=\alpha s.$ From the definition of the $\alpha$-content this condition is satisfied if and only if $\xi_1>s$.
\end{proof}
\end{lem}



The space $\Ps$ satisfies a stability property and a vanishing property with respect to the operators $\B_m$, which will play a key role in the proof of \cref{thm vanishing}.

\begin{prop}[Stability property]\label{prop stability}
For any $s\geq1$, the space $\Ps$ is stable by the operators $\B_m(\bfp,-\alpha s)$ for every $m\geq1.$
\begin{proof}
It is enough to prove that 
\begin{equation}\label{eq thm stability}
  \B_m(\bfp,-\alpha s)\cdot \Jxi(\bfp)\in\Ps  
\end{equation}
for every partition $\xi$ such that $\xi_1\leq s$. Fix such a partition $\xi$.

From \cref{thm ChapuyDolega2022}, we know that
$$t^m \B_m(\bfp,-\alpha s) \cdot \tau_b(t,\bfp,\bfq, \underline{-\alpha s})=\frac{m\partial}{\partial q_m}\tau_b(t,\bfp,\bfq,-\alpha s).$$
By extracting the coefficient of $t^{|\xi|+m}\Jxi(\bfq)$ in the last equation, we get
\begin{align*}
  \B_m(\bfp,-\alpha s) \cdot \frac{\Jxi(\bfp)\Jxi(\underline{-\alpha s})}{j_\xi^{(\alpha)}}
  &=\left[t^{|\xi|+m}J_\xi(\bfq)\right]\frac{m\partial}{\partial q_m}\tau_b(t,\bfp,\bfq,-\alpha s)\\
  &=\sum_{\pi\vdash |\xi|+m}\frac{J_\pi^{(\alpha)}(\bfp)J_\pi^{(\alpha)}(\underline{-\alpha s})}{j^{(\alpha)}_\pi}[J_\xi(\bfq)]\frac{m\partial}{\partial q_m}J_\pi(\bfq).
\end{align*}
Using \cref{lem Ps}, we know that $\frac{\Jxi(\underline{-\alpha s})}{j_\xi^{(\alpha)}}\neq 0$, and that the right hand-side of the last equation is in $\Ps$. This finishes the proof of \cref{eq thm stability} and hence the proof of the proposition.
\end{proof}
\end{prop}

\begin{prop}\label{prop annhilation}
Fix an integer $s\geq 1$. The operator $\exp\left(\Binf(t,\bfp,-\alpha
  s)\right)$ is a well defined operator on $\Ps$ with the domain being
$\Span_{\Ps}\{1,t,\dots,t^{s}\}$. In other terms, for every $\xi$, such that $\xi_1\leq s$ and for every $\ell>s$, we have 
$$[t^\ell]\left(\exp\left(\Binf(t,\bfp,-\alpha s)\right)\cdot \Jxi (\bfp) \right)=0.$$
\begin{proof}
\cref{rmq:WellDef} implies that $\exp\left(\Binf(t,\bfp,-\alpha
  s)\right)$ is well-defined on $\mcP$, and consequencetly it is
well-defined on $\mathbb{Q}(\alpha)[\bfp,\bfq]
\llbracket z \rrbracket$.
In particular, we can investigate the action of
$\exp\left(\Binf(t,\bfp,-\alpha s)\right)$ on $\tau_b(z,\bfp,\bfq,\underline{-\alpha s})$. Using \cref{prop tau} we have
\begin{multline*}
  [t^\ell]\left(\exp\left(\Binf(t,\bfp,-\alpha s)\right)\cdot\tau_b(z,\bfp,\bfq,\underline{-\alpha s}) \right)
  \\=[t^\ell]\left(\exp\left(\Binf(t,\bfp,-\alpha s)\right)\exp\left(\sum_{m\geq1}\frac{z^mq_m}{m}\B_m(\bfp,-\alpha s)\right)\cdot 1 \right).  
\end{multline*}
But since the operators $\left(\B_m(\bfp,-\alpha s)\right)_{m\geq1}$ commute (see \cref{prop comm B}), the last equation can be rewritten as follows:
\begin{align*}
  [t^\ell]\left(\exp\left(\Binf(t,\bfp,-\alpha s)\right)\right.&\left.\cdot\tau_b(z,\bfp,\bfq,\underline{-\alpha s}) \right)\\
  &=\exp\left(\sum_{m\geq1}\frac{z^mq_m}{m}\B_m(\bfp,-\alpha s)\right)
    [t^\ell]\left(\exp\left(\Binf(t,\bfp,-\alpha s)\right)\cdot 1 \right)\\
  &=\exp\left(\sum_{m\geq1}\frac{z^mq_m}{m}\B_m(\bfp,-\alpha s)\right)\cdot[t^\ell]\tau_b(t,\bfp,\underline 1,\underline{-\alpha s}),
\end{align*}
where
$$[t^\ell]\tau_b(t,\bfp,\underline 1,\underline{-\alpha s})=\sum_{\xi\vdash \ell}\frac{\Jxi(\bfp)\Jxi(\underline{1})\Jxi(\underline{ -\alpha s})}{j^{(\alpha)}_\xi}.$$
We claim that 
  \begin{equation}\label{eq poly tau}
      [t^\ell]\tau_b(t,\bfp,\underline 1,\underline{-\alpha s})=0.
  \end{equation}
Since $\ell>s$, we know that any partition of size $\ell$ contains at least one of the two boxes $(s+1,1)$ and $(1,2)$, of respective $\alpha$-content $\alpha s$ and $-1$.
Hence, from \cref{thm Jack formula},  we know that
  $$\frac{J_\xi(\underline{1})J_\xi(\underline{-\alpha s})}{j_\xi^{(\alpha)}}=0$$
  for any partition $\xi$ of size $\ell$.  This proves \cref{eq poly tau}, and as a consequence, 
  $$  [t^\ell]\left(\exp\left(\Binf(t,\bfp,-\alpha s)\right)\cdot\tau_b(z,\bfp,\bfq,\underline{-\alpha s}) \right)=0.$$
  Let $\xi$ be a partition with $\xi_1\leq s$. We extract the coefficient of $z^{|\xi|}\Jxi(\bfq)$ in the last equation, and we use the fact that this extraction commutes with the action of $\exp\left(\Binf(t,\bfp,-\alpha s)\right)$:
  $$[t^\ell]\left(\exp\big(\Binf(t,\bfp,-\alpha s)\big)\cdot[z^{|\xi|}\Jxi(\bfq)]\tau_b(z,\bfp,\bfq,\underline{-\alpha s}) \right)=0.$$
  Hence
  $$[t^\ell]\left(\exp\big(\Binf(t,\bfp,-\alpha s)\big)\cdot\frac{\Jxi(\bfp)\Jxi(\underline{-\alpha s})}{j^{(\alpha)}_\xi}\right)=0.$$
  But from \cref{lem Ps} we know that $\Jxi(\underline{-\alpha s})\neq0$, which concludes the proof.
\end{proof}
\end{prop}

\subsection{Proof of \texorpdfstring{\cref{thm vanishing}}{}}
We now prove the main theorem of this section.
\begin{proof}[Proof of \cref{thm vanishing}]
Fix a partition $\lambda$ and an integer $m>|\lambda|$, and let us prove that
$$[t^m]F(\lambda)=0.$$ Let $\ell$ denote the length of $\lambda$.
\cref{eq F,eq Flambda} imply that
\begin{multline}\label{eq thm vanishing}
[t^m]F(\lambda)=\sum_{n_1,\dots,n_\ell\geq 1\atop n_1+\cdots+n_\ell=m}\left([t^{n_1}]\exp\big(\Binf(-t,\bfp,-\alpha \lambda_{1})\big)\right)\\
\cdots  \left([t^{n_\ell}]\exp\big(\Binf(-t,\bfp,-\alpha \lambda_{m})\big)\right)\cdot 1.
\end{multline}
Let $1\leq i\leq \ell$ be such that $n_i>\lambda_i$ (such an integer exists 
because $m>|\lambda|$).
Since $\lambda_1 \geq \cdots \geq \lambda_\ell$, we have a chain of
subspaces $1\in\mcP_{\leq \lambda_\ell} \subset \mcP_{\leq
  \lambda_{\ell-1}}\subset \cdots \subset \mcP_{\leq
  \lambda_{i+1}}$. Therefore \cref{prop stability} implies that 
\begin{equation*}
\bigl([t^{n_{i+1}}]\exp\big(\Binf(-t,\bfp,-\alpha
\lambda_{i+1})\big)\bigr)\cdots\bigl([t^{n_\ell}]\exp\big(\Binf(-t,\bfp,-\alpha
\lambda_{\ell})\big)\bigr)\cdot 1\in\mcP_{\leq\lambda_{i+1}} \subset \mcP_{\leq\lambda_{i}}.
\end{equation*}

We now apply \cref{prop annhilation} with  $s=\lambda_{i}$ and
$\ell=n_i$, and we get
\begin{equation}\label{eq 2 thm vanishing}
\bigl([t^{n_{i}}]\exp\big(\Binf(-t,\bfp,-\alpha \lambda_{i})\big)\bigr)\cdots\bigl([t^{n_\ell}]\exp\big(\Binf(-t,\bfp,-\alpha \lambda_{m})\big)\bigr)\cdot 1=0.
\end{equation}
As a consequence, the contribution of each term in the RHS of \cref{eq thm vanishing} is zero, which concludes the proof of the theorem.
\end{proof}

\section{The commutation relations of operators $\C_\ell$}\label{sec com rel}
The main purpose of this section is to prove \cref{thm com C}. 
\subsection{The operators $Y_{\ell,k}$}
We consider the following catalytic version of the operators $\C_{\ell,k}$ defined in \cref{ssec proof expr F}. If $\ell$ and $k$ are two integers, then $Y_{\ell,k}$ is defined by
$$Y_{\ell,k}:=\left\{\begin{array}{cc}
    [u^\ell]\left(\GY+uY_+\right)^{\ell+k} &  \text{ if } \ell,k\geq0,\\
    0 & \text{ otherwise.}
\end{array}\right.$$ on $\PY$. 
In practice, we think of $Y_{\ell,k}$ as  the sum of all successions of operators $Y_+$ and $\GY$ in which $Y_+$ appears $\ell$ times and $\GY$ appears $k$ times.
$Y_{\ell,k}$ and $\C_{\ell,k}$ are related by
\begin{equation}\label{eq C Y}
    \C_{\ell,k}=\Theta_Y Y_{\ell,k}\frac{y_0}{1+b}.
\end{equation}

These operators satisfy the following recursive relations.
\begin{lem}
Fix a pair of integers $(\ell,k)$.
One has
\begin{equation}\label{eq recursion Y}
  Y_{\ell,k}=Y_{\ell,k-1}\GY+Y_{\ell-1,k}Y_++\delta_{\ell,0}\delta_{k,0},
\end{equation}
where $\delta$ denotes the Kronecker delta.
Moreover, if $1\leq m\leq \ell$ then 
\begin{equation}\label{eq recursion Y 2}
  Y_{\ell,k}=\sum_{0\leq i\leq k}Y_{m-1,i}Y_+Y_{\ell-m,k-i},
\end{equation}
and if $1\leq m\leq \ell+k$, then 
\begin{equation}\label{eq recursion Y 3}
    Y_{\ell,k}=\sum_{0\leq j\leq m}Y_{j,m-j}Y_{\ell-j,k-m+j}.
\end{equation}

\begin{proof}
If $\ell\leq 0$ or $k\leq 0$ then the equations are immediate from the definition.
Let us suppose that $\ell> 0$ and $k> 0$.
In order to obtain \cref{eq recursion Y}, we expand $Y_{\ell,k}$ according to the rightmost operator; the sum of terms ending with $\GY$ (reps. $Y_+$) give $Y_{\ell,k-1}\GY$ (resp. $Y_{\ell-1,k}Y_+$).

Similarly, we obtain \cref{eq recursion Y 2} by expanding $Y_{\ell,k}$ according to the position of the $m$-th occurrence of the operator $Y_+$, and we obtain \cref{eq recursion Y 3} by expanding on the number of occurrences of $Y_+$ in the $m$ left operators. 
\end{proof}
\end{lem}


\subsection{Catalytic operators in \texorpdfstring{$\tY$}{} and \texorpdfstring{$\tZ$}{}}
We consider three new alphabets $$Y':=\{y'_0,y'_1\dots\},\hspace{1cm}Z:=\{z_0,z_1,\dots\}, \hspace{0.4cm}\text{and}\hspace{0.4cm} Z':=\{z'_0,z'_1,\dots\}.$$
We also denote 
$$\tY:=Y\cup Y', \hspace{0.4cm} \text{and}\hspace{0.4cm}\tZ:=Z\cup Z'.$$  
Let $\PYZ$ be the space 
$$\PYZ=\Span_{\mathbb{Q}(b)}\left\{y_iz_jp_\lambda,y'_iz'_jp_\lambda\right\}_{i,j\in\mathbb{N},\lambda\in\mathbb{Y}}.$$

In this section, we use several differential operators acting on the space $\PYZ$ which have been introduced in \cite{ChapuyDolega2022}. 




We define the operators $Y'_+$ and $\Gamma_{Y'}$ by replacing $y_i$ by $y'_i$ in \cref{eq Y+} and \cref{eq LambdaY}, respectively. Similarly, we define $Z_+$, $Z'_+$, $\Gamma_{Z}$ and $\Gamma_{Z'}$. 
We also consider the catalytic operators in the two variables $\tY$ and $\tZ$.
$$\Gamma^{Y,Y'}_{Z,Z'}=(1+b)\cdot\sum_{i,j,k\geq 1}\frac{y'_{i+j-1}z'_{k} \partial^2}{\partial y_{i+k-1}\partial z_{j-1}} +\sum_{i,j,k\geq1} \frac{y_{i+j-1}z_{k} \partial^2}{\partial y'_{i+k-1}\partial z'_{j-1}}+b\cdot \sum_{i,j,k\geq1}\frac{y'_{i+j-1} z'_{k}\partial^2}{\partial y'_{i+k-1}\partial z'_{j-1}},$$

$$\GtZ=\Gamma_Z+\Gamma_{Z'}+\Gamma_{Z,Z'}^{Y,Y'},\qquad \text{and} \qquad \tZ_+=Z_+ +Z'_+.$$
We also consider the following operator
$$\Theta _{\tZ}=\sum_{i\geq0} p_i\frac{\partial}{\partial
  z_i}+\sum_{i,j\geq0}y_{i+j}\frac{\partial^2}{\partial y'_i \partial
  z'_j}\colon \PYZ \rightarrow \PY.$$

Similarly, the operators $\Gamma^{Z,Z'}_{Y,Y'}$, $\Gamma_{\tY}$, $\tY_+$, $\Theta_{\tY}$ are defined by exchanging $z_i\leftrightarrow y_i$ and $z'_i\leftrightarrow y'_i$ in the previous definitions.
Moreover, let $\Delta$ be the operator
$$\Delta:=(1+b)\cdot \sum_{i,j\geq0}\frac{y'_jz'_i\partial^2}{\partial y_i\partial z_j}+\sum_{i,j\geq0}\frac{y_jz_i\partial^2}{\partial y'_i\partial z'_j}+b\cdot\sum_{i,j\geq 0}\frac{y'_jz'_i\partial^2}{\partial y'_i\partial z'_j}.$$

We now consider a two catalytic variables version of $Y_{\ell,k}$, defined as the operators on $\PYZ$ given by
\[\widetilde{Y}_{\ell,k}:=\begin{cases}[u^\ell]\left(\GtY+u \tY_+\right)^{\ell+k} &\text{ if }\ell,k\geq0,  \\
     0 &\text{ otherwise }\end{cases},\ \ \widetilde{Z}_{\ell,k}:=\begin{cases}[u^\ell]\left(\GtZ+u \tZ_+\right)^{\ell+k} &\text{ if }\ell,k\geq0,  \\
     0 &\text{ otherwise }\end{cases},\]
Finally, we define the operator on $\PY$: 
\begin{equation}
  \label{eq:CYDef}
  \C^Y_{\ell,k}:=\Theta_{\widetilde{Z}}\tZ_{\ell,k}\frac{z_0}{1+b}.
  \end{equation}

\subsection{Preliminary commutation relations} 
In this section, we prove some commutation relations satisfied by
these operators. We will use some identities from the work of Chapuy
and the second author~\cite{ChapuyDolega2022}, which are stated for the operators $\Lambda_{\tilde{Y}},
\Lambda_{\tilde{Z}}$ that are related to our
$\Gamma_{\tilde{Y}}, \Gamma_{\tilde{Z}}$ by
\begin{equation}
  \label{eq:CDPomoc}
  \Gamma_{\tilde{Y}} =
\tY_+\Lambda_{\tilde{Y}}, \quad \Gamma_{\tilde{Z}} =
\tZ_+\Lambda_{\tilde{Z}}.
\end{equation}

\begin{lem}\label{lem Delta}
We have the following equalities between operators on $\PYZ$
$$\Delta \GtY=\GtZ\Delta,\ \Delta \GtZ=\GtY\Delta,\qquad \text{ and }
\qquad \Delta \tY_+=\tZ_+ \Delta,\ \Delta \tZ_+=\tY_+ \Delta.$$
As a consequence, for every $\ell,k\geq0$, we have
\begin{equation}\label{eq com Delta Y Z}
  \Delta \tY_{\ell,k}=\tZ_{\ell,k}\Delta,\ \Delta \tZ_{\ell,k}=\tY_{\ell,k}\Delta.   
\end{equation}
\begin{proof}
Not that each identity comes in pair with an identity obtained by exchanging
the alphabets $y_i \leftrightarrow z_i$, $y_i' \leftrightarrow z_i'$. Therefore it is enough to prove
only the first identity for each pair (this argument will appear
all over this section, so we will always prove only one of the
identities that appear in such pairs). The second identity is direct
from the definitions, and the first one follows from the second one,
\cref{eq:CDPomoc}, and the identity $\Lambda_{\tilde{Z}}\Delta = \Delta
\Lambda_{\tilde{Y}}$ from~\cite[Eq. (29a)]{ChapuyDolega2022}. \cref{eq com Delta Y Z} follows immediately.
\end{proof}
\end{lem}
We have the following commutation relations between $\tY$ and $\tZ$ operators.
\begin{lem}\label{lem prel 2}\cite{ChapuyDolega2022}
We have the following commutation relations on $\mathcal P_{\tY,\tZ}$,
\begin{equation}
\left[\tZ_+,\tY_+\right]=0, \qquad \left[\GtZ,\GtY\right]=0,
\end{equation}
\begin{equation}
\left[ \GtZ,\tY_+\right]=-\left[ \tZ_+,\GtY\right]=\tY_+\Delta \tY_+.
\end{equation}
\begin{proof}
  \cite[Eq. (30)]{ChapuyDolega2022} says that for $m,n \geq 0$ the following equation
  holds
\[[\tilde{Z}_+\Lambda_{\tilde{Z}}^{m+1},\tilde{Y}_+\Lambda_{\tilde{Y}}^{n}] +
  [\tilde{Z}_+\Lambda_{\tilde{Z}}^{n+1},\tilde{Y}_+\Lambda_{\tilde{Y}}^{m}]
  =\tilde{Y}_+\Lambda_{\tilde{Y}}^n\tilde{Z}_+\Lambda_{\tilde{Z}}^m\Delta +
  \tilde{Y}_+\Lambda_{\tilde{Y}}^m\tilde{Z}_+\Lambda_{\tilde{Z}}^n\Delta.\]
Special cases of this equation ($m=n=0$,
$m=n=1$ and $(m,n)=(1,0)$ resp.), \cref{eq:CDPomoc}, and \cref{lem Delta} finish the proof.
\end{proof}
\end{lem}

\begin{lem}\label{lem prel 1}
The following operators are equal:
\begin{subequations}
\begin{equation}\label{eq com lem 1}
    \Theta_{\tZ}\tY_{i,j}=Y_{i,j}\Theta_{\tZ}\colon \PYZ
    \rightarrow\PY,\ \ \      \Theta_{\tY}\tZ_{i,j}=Z_{i,j}\Theta_{\tY}\colon \PYZ
    \rightarrow\PZ\text{ for } i,j\geq0,
\end{equation}
\begin{equation}\label{eq com lem 2}
    \Theta_{Y}\Theta_{\tZ}=\Theta_Z\Theta_{\tY}\colon \PYZ \rightarrow
    \mcP,
\end{equation}
\begin{equation}\label{eq com lem 3}
    \Theta_{\tY}z_i=z_i\Theta_{Y}\colon \PY\rightarrow\PZ,\ \ \ \Theta_{\tZ}y_i=y_i\Theta_{z}\colon \PZ\rightarrow\PY \text{ for } i\geq0,
\end{equation}
\begin{equation}\label{eq com lem 4}
    \Theta_{\tZ}\Delta\frac{z_0}{1+b}=1\colon\PY\rightarrow\PY,\ \ \     \Theta_{\tY}\Delta\frac{y_0}{1+b}=1\colon\PZ\rightarrow\PZ,
\end{equation}    
\begin{equation}\label{eq com lem 5}
        \tY_{i,j}z_0=z_0Y_{i,j}\colon \PY
        \rightarrow\PYZ,\ \ \ \tZ_{i,j}y_0=y_0Z_{i,j}\colon \PZ
        \rightarrow\PYZ, \text{ for } i,j\geq0,
\end{equation}
\begin{equation}\label{eq com lem 6}
        \C^Y_{i,j}\frac{y_0}{1+b}=\frac{y_0}{1+b}\C_{i,j} \text{ for } i,j\geq0, \text{ as operators from $\mcP$ to $\PY$}.
\end{equation}
\end{subequations}

\begin{proof}
  \cref{eq com lem 1} is a consequence of \cref{eq:CDPomoc} and the identities
  \[ \Theta_{\tilde{Z}}\Lambda_{\tilde{Y}} =
                                                    \Lambda_{Y}\Theta_{\tilde{Z}},\
                                                    \Theta_{\tilde{Y}}\Lambda_{\tilde{Z}}
                                                    =
                                                    \Lambda_{Z}\Theta_{\tilde{Y}},\
                                                    \ \
                                                    \Theta_{\tilde{Z}}\tilde{Y}_+
                                                    =
                                                    Y_+\Theta_{\tilde{Z}},\
                                                  \Theta_{\tilde{Y}}\tilde{Z}_+= Z_+\Theta_{\tilde{Y}} \]
  proved in~\cite[Eqs. (29e) and (29f)]{ChapuyDolega2022}. Eq. \eqref{eq com lem 2}---\eqref{eq com lem 5}  are direct from the definitions. 
\cref{eq com lem 5}
gives $$\tZ_{i,j}\frac{y_0}{1+b}=\frac{y_0}{1+b}Z_{i,j}\colon \PZ\rightarrow\PYZ.$$
\cref{eq:CYDef} implies that applying $\Theta_{\tZ}$ on the left and
$\frac{z_0}{1+b}$ on the right we get
$$\C^Y_{i,j}\frac{y_0}{1+b}=\Theta_{\tZ}\frac{y_0}{1+b}Z_{i,j}\frac{z_0}{1+b}.$$
We deduce
\cref{eq com lem 6} from \cref{eq com lem
  3}, and definition of $\C_{i,j}$ (see~\cref{eq C Y}).
\end{proof}
\end{lem}

We conclude this section with the following lemma.
\begin{lem}\label{lem C CY}
Let $\ell,k\geq0$. Then,
$$\Theta_Y\C^Y_{\ell,k}=\C_{\ell,k}\Theta_Y, \text{as operators from  }\PY \text{ to }\mcP.$$
\begin{proof}
Applying Eqs. \eqref{eq com lem 2}, \eqref{eq com lem 1} and \eqref{eq com lem 3} successively, we get that
\begin{align*}
\Theta_Y\C^Y_{\ell,k}
&=\Theta_Y\Theta_{\widetilde{Z}}\tZ_{\ell,k}\frac{z_0}{1+b}\\
&=\Theta_Z\Theta_{\widetilde{Y}}\tZ_{\ell,k}\frac{z_0}{1+b}\\
&=\Theta_Z Z_{\ell,k}\Theta_{\widetilde{Y}}\frac{z_0}{1+b}\\
&=\Theta_Z Z_{\ell,k}\frac{z_0}{1+b}\Theta_{Y}\\
&=\C_{\ell,k}\Theta_{Y}.\qedhere
\end{align*}
\end{proof}
\end{lem}

\subsection{Proof of \texorpdfstring{\cref{thm com C}}{}}
In this section, we prove the following theorem.
\begin{thm}\label{thm com C}
Let $m>0$. Then 
$$\left[\C_\ell,\C_m\right]=\left\{\begin{array}{cc}
     0 & \text{ if } \ell>0,  \\
    (m+1)\C_{m+1} & \text{ if }  \ell=0.
\end{array}\right.$$
\end{thm}

The idea of the
proof is to start from \cref{lem prel 2} which computes the commutator
of a linear monomial in $\tY_+,\GtY$ with a linear monomial in
$\tZ_+,\GtZ$, and use inductions to obtain the commutators of such
monomials of arbitrary degrees.

The proof is organized as follows: we start from  \cref{lem prel 2} which gives an expression for the commutator $[\tZ_{\ell,k},\tY_+]$ and $[\tZ_{\ell,k},\GtZ]$ when $\ell+k=1$. By induction we obtain in \cref{lem com tZ tY} an expression for these commutators for any $\ell$ and $k$. By "forgetting" the first catalytic operator $\tZ$, we deduce in \cref{cor com C Y} the commutators $[\C^Y_{\ell,k},\tY_+]$. We then use induction to obtain an expression for $\sum_{0\leq i\leq k}\left[\frac{1}{\ell+i}\C^Y_{\ell,i},Y_{m,k-i}\right]$.
Finally, we deduce \cref{thm com C} by forgetting the catalytic variable $Y$.

\begin{lem}\label{lem com tZ tY}
For any integers $\ell,k\geq -1$, we have the following equalities between operators on $\PYZ$.
\begin{equation}\label{eq tilde Z Y 1}
\left[\tZ_{\ell,k},\widetilde{Y}_{+}\right]=\sum_{i=0}^\ell\sum_{j=0}^{k-1}\tY_+ \tY_{i,j}\Delta\tY_{\ell-i,k-1-j}\tY_+,    
\end{equation}
and
\begin{equation}\label{eq tilde Z Y 2}
  \left[\tZ_{\ell,k},\GtY\right]=-\sum_{i=0}^{\ell-1}\sum_{j=0}^k\tY_+
  \tY_{i,j}\Delta\tY_{\ell-1-i,k-j}\tY_+.  
\end{equation}

In particular, by exchanging the variables $\tY$ and $\tZ$, we get that 
\begin{equation}\label{eq tilde Y Z}
    \left[\tY_{\ell-1,k},\tZ_+ \right]+\left[ \tY_{\ell,k-1},\GtZ \right ]=0, 
\end{equation}

\begin{proof}

We prove simultaneously the three equations by induction on $\ell+k$. If $\ell=-1$ or $k=-1$ or $\ell=k=0$ the result is immediate from the definitions. For $(\ell=0,k=1)$ and $(\ell=1,k=0)$ it corresponds to \cref{lem prel 2}.

We now fix $\ell, k\geq 0$ such that $(\ell,k)\neq (0,0)$ and we suppose that \cref{eq tilde Z Y 1,eq tilde Z Y 2} hold for all  $(i,j)$ such that $i+j<\ell+k$.
First, since $(\ell,k)\neq(0,0)$, we have an analogue of \cref{eq recursion Y}:
\begin{equation} \label{eq induction}
\tZ_{\ell,k}=\tZ_{\ell,k-1}\GtZ + \tZ_{\ell-1,k}\tZ_+
\end{equation}
Applying \cref{eq tilde Z Y 1} with the pairs $(\ell,k-1)$ and $(\ell-1,k)$, and using \cref{lem prel 2} we get that 

\begin{align}
  \label{eq:CommutPomoc}
    \left[\tZ_{\ell,k},\widetilde{Y}_{+}\right]
    &=\left[ \tZ_{\ell,k-1}\GtZ,\tY_+\right]+
    \left[  \tZ_{\ell-1,k}\tZ_+,\tY_+\right]\nonumber\\
    &=\sum_{i=0}^\ell\sum_{j=0}^{k-2}\tY_+ \tY_{i,j}\Delta\tY_{\ell-i,k-2-j}\tY_+ \GtZ+  \tZ_{\ell,k-1}\tY_+\Delta \tY_+\nonumber\\
    &\quad+\sum_{i=0}^{\ell-1}\sum_{j=0}^{k-1}\tY_+ \tY_{i,j}\Delta\tY_{\ell-1-i,k-1-j}\tY_+\tZ_+\nonumber\\
    &=\sum_{i=0}^\ell\sum_{j=0}^{k-2}\tY_+ \tY_{i,j}\Delta\tY_{\ell-i,k-2-j}\GtZ\tY_+ -\sum_{i=0}^\ell\sum_{j=0}^{k-2}\tY_+ \tY_{i,j}\Delta\tY_{\ell-i,k-2-j}\tY_+\Delta\tY_+\nonumber\\
    &\quad+\tZ_{\ell,k-1}\tY_+\Delta \tY_+
    +\sum_{i=0}^{\ell-1}\sum_{j=0}^{k-1}\tY_+ \tY_{i,j}\Delta\tY_{\ell-1-i,k-1-j}\tZ_+\tY_+.
    \end{align}
Fix $0\leq i\leq \ell$ and $0\leq j \leq k-1$. \cref{eq
  tilde Y Z}, and \cref{lem Delta} show that the operator $\tY_+\tY_{i,j}\Delta\tY_{\ell-i,k-2-j}\GtZ\tY_+$ is equal to:
\[ \tY_+
  \tY_{i,j} \GtY\Delta \tY_{\ell-i,k-2-j}\tY_+- \tY_+
  \tY_{i,j}\Delta \tY_{\ell-i-1,k-1-j}\tZ_+\tY_++\tY_{i,j}\tY_+\Delta \tY_{\ell-i-1,k-1-j}\tY_+.\]
Therefore, we get that the sum of the first and the fourth item in the
RHS of \cref{eq:CommutPomoc} is equal to 
$$\sum_{i=0}^\ell\sum_{j=0}^{k-2}\tY_+ \tY_{i,j}\GtY\Delta\tY_{\ell-i,k-2-j}\tY_+ +\sum_{i=0}^{\ell-1}\sum_{j=0}^{k-1}\tY_+ \tY_{i,j}\tY_+\Delta\tY_{\ell-1-i,k-1-j}\tY_+.$$
On the other hand, applying \cref{eq tilde Z Y 1} with the pair $(\ell,k-1)$, we obtain that 
$$\tZ_{\ell,k-1}\tY_+\Delta \tY_+=\tY_+\Delta\tY_{\ell,k-1} \tY_+ + \sum_{i=0}^{\ell}\sum_{j=0}^{k-2}\tY_+\tY_{i,j}\Delta\tY_{\ell-i,k-2-j}\tY_+\Delta\tY_+.$$
Hence 
    \begin{align*}
    \left[\tZ_{\ell,k},\widetilde{Y}_{+}\right]
    &=\sum_{i=0}^\ell\sum_{j=0}^{k-2}\tY_+ \tY_{i,j}\GtY\Delta\tY_{\ell-i,k-2-j}\tY_+  +\tY_+\Delta\tY_{\ell,k-1} \tY_+\\
    &+\sum_{i=0}^{\ell-1}\sum_{j=0}^{k-1}\tY_+ \tY_{i,j}\tY_+\Delta\tY_{\ell-1-i,k-1-j}\tY_+.
\end{align*}
Shifting the summation indices we get
\begin{align*}\label{eq com tZ tY}
\left[\tZ_{\ell,k},\widetilde{Y}_{+}\right]
&=\sum_{i=0}^\ell\sum_{j=1}^{k-1}\tY_+ \tY_{i,j-1}\GtY\Delta\tY_{\ell-i,k-1-j}\tY_+ 
+\tY_+\Delta\tY_{\ell,k-1} \tY_+\\
&+ \sum_{i=1}^{\ell}\sum_{j=0}^{k-1}\tY_+ \tY_{i-1,j}\tY_+\Delta\tY_{\ell-i,k-1-j}\tY_+ \\
&=\sum_{i=0}^\ell\sum_{j=0}^{k-1}\tY_+ \tY_{i,j-1}\GtY\Delta\tY_{\ell-i,k-1-j}\tY_+ 
+\tY_+\Delta\tY_{\ell,k-1} \tY_+\\
&+ \sum_{i=0}^{\ell}\sum_{j=0}^{k-1}\tY_+ \tY_{i-1,j}\tY_+\Delta\tY_{\ell-i,k-1-j}\tY_+,
\end{align*}
where the second equality follows from the fact that $\tY_{i,j}=0$ if $i<0$ or $j<0$.
 For each couple of indices $(i,j)\neq(0,0)$, we regroup the terms in the two sums of the last equation by applying \cref{eq induction}. On the other hand, note that the second term in the last equation can be written $\tY_+\tY_{0,0}\Delta\tY_{\ell,k-1} \tY_+$. We deduce that 
$$\left[\tZ_{\ell,k},\widetilde{Y}_{+}\right]=\sum_{i=0}^\ell\sum_{j=0}^{k-1}\tY_+ \tY_{i,j}\Delta\tY_{\ell-i,k-1-j}\tY_+.$$
We prove \cref{eq tilde Z Y 2} in a similar way. Using \cref{eq induction} and the induction hypothesis, we have

\begin{align*}
    \left[\tZ_{\ell,k},\GtY\right]
    &=\left[ \tZ_{\ell,k-1}\GtZ,\GtY\right]+\left[ \tZ_{\ell-1,k}\tZ_+,\GtY\right]\\
    &=-\sum_{i=0}^{\ell-1}\sum_{j=0}^{k-1}\tY_+ \tY_{i,j}\Delta\tY_{\ell-1-i,k-1-j}\tY_+ \GtZ-\sum_{i=0}^{\ell-2}\sum_{j=0}^{k}\tY_+ \tY_{i,j}\Delta\tY_{\ell-2-i,k-j}\tY_+\tZ_+\\
    &\hspace{0.3cm} -  \tZ_{\ell-1,k}\tY_+\Delta \tY_+.
    \end{align*}
    From \cref{lem prel 2}, and \eqref{eq tilde Z Y 1} we have
    
    \begin{align*}
    \left[\tZ_{\ell,k},\GtY\right]
    &=-\sum_{i=0}^{\ell-1}\sum_{j=0}^{k-1}\tY_+ \tY_{i,j}\Delta\tY_{\ell-1-i,k-1-j}\GtZ\tY_+ +\sum_{i=0}^{\ell-1}\sum_{j=0}^{k-1}\tY_+ \tY_{i,j}\Delta\tY_{\ell-1-i,k-1-j}\tY_+\Delta\tY_+\\
    &\hspace{0.3cm}-\sum_{i=0}^{\ell-2}\sum_{j=0}^{k}\tY_+ \tY_{i,j}\Delta\tY_{\ell-2-i,k-j}\tZ_+\tY_+ -\tY_+ \Delta \tY_{\ell-1,k}\tY_+\\ 
    &\hspace{0.3cm} -\sum_{i=0}^{\ell-1}\sum_{j=0}^{k-1}\tY_+ \tY_{i,j}\Delta\tY_{\ell-1-i,k-1-j}\tY_+\Delta \tY_+.
    \end{align*}
Applying \cref{eq tilde Y Z} with $(\ell-i,k-1-j)$ for $0\leq i\leq \ell$ and $0\leq j \leq k-1$, we get that 
\begin{align*}
\left[\tZ_{\ell,k},\GtY\right]
    &=-\sum_{i=0}^{\ell-1}\sum_{j=0}^{k-1}\tY_+ \tY_{i,j}\GtY\Delta\tY_{\ell-1-i,k-1-j}\tY_+ -\sum_{i=0}^{\ell-2}\sum_{j=0}^{k}\tY_+ \tY_{i,j}\tY_+\Delta\tY_{\ell-2-i,k-j}\tY_+ \\
    &\hspace{0.3cm}-\tY_+ \Delta \tY_{\ell-1,k}\tY_+\\
    &=-\sum_{i=0}^{\ell-1}\sum_{j=0}^{k}\tY_+ \tY_{i,j}\Delta\tY_{\ell-1-i,k-j}\tY_+\qedhere
\end{align*}

\end{proof}

\end{lem}

We deduce the following corollary.
\begin{cor}\label{cor com C Y}
Let $\ell,k\geq 0$. As operators on $\PY$,
$$\left[\C^Y_{\ell,k},Y_+\right]=(\ell+k)Y_+Y_{\ell,k-1}Y_+, \qquad  \left[\C^Y_{\ell,k},\GY\right]=-(\ell+k)Y_+Y_{\ell-1,k}Y_+.$$
\begin{proof}
We start by multiplying \cref{eq tilde Z Y 1} by $\Theta_{\tZ}$ on the left and $\frac{z_0}{1+b}$ on the right, and we use Equations \eqref{eq com lem 1}, \eqref{eq com lem 5} and \eqref{eq com lem 4} to obtain:
\begin{align*}
  \left[\C^Y_{\ell,k},Y_+\right]
  &=\Theta_{\tZ}\sum_{i=0}^\ell\sum_{j=0}^{k-1}\tY_+ \tY_{i,j}\Delta\tY_{\ell-i,k-1-j}\tY_+\frac{z_0}{1+b}\\
  &=\sum_{i=0}^\ell\sum_{j=0}^{k-1}Y_+ Y_{i,j}\Theta_{\tZ}\Delta\frac{z_0}{1+b} Y_{\ell-i,k-1-j} Y_+\\
  &=\sum_{i=0}^\ell\sum_{j=0}^{k-1}Y_+ Y_{i,j} Y_{\ell-i,k-1-j} Y_+.\\
  &=\sum_{m=0}^{\ell+k-1}\sum_{i=0}^\ell Y_+ Y_{i,m-i} Y_{\ell-i,k-1-m+i} Y_+.
\end{align*}
From \cref{eq recursion Y 3}, we know that in the last line, for each $m$, the second sum is equal to $Y_+Y_{\ell,k-1}Y_+$, which concludes the proof of the first equation.
Similarly, we obtain the second equation of the corollary from \cref{eq tilde Z Y 2}.
\end{proof}
\end{cor}

We deduce the following proposition.
\begin{prop}\label{prop com C Y}
Fix $m\geq0$ and $k\geq -1$. If $\ell>0$ then 
\begin{equation}\label{eq 1 prop com C}
  \sum_{0\leq i\leq k}\left[\frac{1}{\ell+i}\C^Y_{\ell,i},Y_{m,k-i}\right]=-Y_{\ell+m+1,k-1}.  
\end{equation}
Moreover, if $\ell=0$ then  
\begin{equation}\label{eq 2 prop com C}
    \sum_{1\leq i\leq k}\left[\frac{1}{i}\C^Y_{0,i},Y_{m,k-i}\right]=mY_{m+1,k-1}.
\end{equation}

\begin{proof}
We proceed by induction on $k+m$.
For $k=-1$ the two equations are immediate from the definitions. 

Let us start by proving \cref{eq 1 prop com C}. 
Fix $(k,m)$ with $k\geq0$, such that \cref{eq 1 prop com C} is satisfied for every $(j,s)$ such that $s+j< k+m$. 
Fix $0\leq i\leq k$. We rewrite \cref{eq recursion Y} as follows.
$$Y_{m,k-i}=\mathbbm{1}_{m>0}Y_{m-1,k-i}Y_++Y_{m,k-i-1}\GY+\delta_{m,0}\delta_{k,i}.$$
Hence, using \cref{cor com C Y}, we get  for each $0\leq i\leq k$
\begin{multline*}
     \left[\frac{1}{\ell+i}\C^Y_{\ell,i},Y_{m,k-i}\right]= \mathbbm{1}_{m>0}\left[\frac{1}{\ell+i}\C^Y_{\ell,i},Y_{m-1,k-i}\right]Y_+ + \mathbbm{1}_{m>0}Y_{m-1,k-i}Y_+Y_{\ell,i-1}Y_+\\ 
     +\left[\frac{1}{\ell+i}\C^Y_{\ell,i},Y_{m,k-i-1}\right]\GY - Y_{m,k-1-i}Y_+Y_{\ell-1,i}Y_+.
\end{multline*}

Applying the induction hypothesis on the pairs $(m-1,k)$ and $(m,k-1)$, we get that 
\begin{multline*}
    \sum_{0\leq i\leq k}\left[\frac{1}{\ell+i}\C^Y_{\ell,i},Y_{m,k-i}\right]
    =-\mathbbm{1}_{m>0}Y_{\ell+m,k-1}Y_+
    +\mathbbm{1}_{m>0}\sum_{0\leq i\leq k}Y_{m-1,k-i}Y_+Y_{\ell,i-1}Y_+\\
    -Y_{\ell+m+1,k-2}\GY
    -\sum_{0\leq i \leq k}Y_{m,k-1-i}Y_+Y_{\ell-1,i}Y_+.
\end{multline*}
Using \cref{eq recursion Y 2}, we know that the two sums in the right-hand side of the last equality are both equal to $Y_{\ell+m,k-1}Y_+$. On the other hand, from \cref{eq recursion Y}, we know that
$$Y_{\ell+m+1,k-2}\GY+Y_{\ell+m,k-1}Y_+=Y_{\ell+m+1,k-1},$$
which concludes the proof of \cref{eq 1 prop com C}.

We now prove \cref{eq 2 prop com C} in a similar way. Let $(k,m)$ be two non-negative integers such that \cref{eq 2 prop com C} is satisfied for every $(j,s)$ such that $s+j< k+m$.
For each $1\leq i\leq k$, one has 
\begin{multline*}
     \left[\frac{1}{i}\C^Y_{0,i},Y_{m,k-i}\right]
     = \mathbbm{1}_{m>0}\left[\frac{1}{i}\C^Y_{0,i},Y_{m-1,k-i}\right]Y_+ + \mathbbm{1}_{m>0}Y_{m-1,k-i}Y_+Y_{0,i-1}Y_+\\ 
     +\left[\frac{1}{i}\C^Y_{0,i},Y_{m,k-1-i}\right]\GY.
\end{multline*}

Applying the induction hypothesis on the pairs $(m-1,k)$ and $(m,k-1)$, we get that 
\begin{multline*}
    \sum_{1\leq i\leq k}\left[\frac{1}{i}\C^Y_{0,i},Y_{m,k-i}\right]
    =(m-1)\mathbbm{1}_{m>0}Y_{m,k-1}Y_+
    +\mathbbm{1}_{m>0}\sum_{0\leq i\leq k}Y_{m-1,k-i}Y_+Y_{0,i-1}Y_+\\
    \quad+mY_{m+1,k-2}\GY.
\end{multline*}
But from \cref{eq recursion Y 2}, we know that the sum in the right hand is equal to $Y_{m,k-1}Y_+$.
Hence
\begin{align*}
\sum_{1\leq i\leq k}\left[\frac{1}{i}\C^Y_{0,i},Y_{m,k-i}\right]
    &=m\mathbbm{1}_{m>0}Y_{m,k-1}Y_++mY_{m+1,k-2}\GY\\
    &=mY_{m,k-1}Y_++mY_{m+1,k-2}\GY\\
    &=mY_{m+1,k-1},
\end{align*}
where we use \cref{eq recursion Y} to obtain the last equality.
\end{proof}
\end{prop}

We deduce the following corollary by forgetting the catalytic variable $Y$.
\begin{cor}\label{cor com C}
Fix $m,k\geq0$. We have the following equalities between operators on $\mcP$:
\begin{equation}\label{eq 3 prop com C}
\sum_{0\leq i\leq k}\left[\frac{1}{\ell+i}\C_{\ell,i},\C_{m,k-i}\right]=-\C_{\ell+m+1,k-1}, \text{ if $\ell>0$},
\end{equation}
and
\begin{equation}\label{eq 2 cor com C}
    \sum_{1\leq i\leq k}\left[\frac{1}{i}\C_{0,i},\C_{m,k-i}\right]=m\C_{m+1,k-1}.
\end{equation}
\begin{proof}
Let us prove \cref{eq 3 prop com C}. Starting from \cref{eq 1 prop com C}, and applying $\Theta _Y$ on the left and $\frac{y_0}{1+b}$ on the right, we get
$$\sum_{0\leq i\leq k}\frac{1}{\ell+i}\left(\Theta_Y \C^Y_{\ell,i} Y_{m,k-i}\frac{y_0}{1+b}-\Theta_Y Y_{m,k-i} \C^Y_{\ell,i}\frac{y_0}{1+b}\right)=-\Theta_Y Y_{\ell+m+1,k-1}\frac{y_0}{1+b}.  
$$
Using \cref{lem C CY} and \cref{eq com lem 6} we obtain
$$\sum_{0\leq i\leq k}\frac{1}{\ell+i}\left( \C_{\ell,i}\Theta_Y Y_{m,k-i}\frac{y_0}{1+b}-\Theta_Y Y_{m,k-i}\frac{y_0}{1+b} \C_{\ell,i}\right)=-\Theta_Y Y_{\ell+m+1,k-1}\frac{y_0}{1+b}.  $$
We deduce \cref{eq 3 prop com C} using \cref{eq C Y}. We obtain in a
similar way \cref{eq 2 cor com C} from \cref{eq 2 prop com C}. 
\end{proof}
\end{cor}

We now prove the main result of this section.
\begin{proof}[Proof of \cref{thm com C}]
Let $\ell,m>0$. Recall the relation between $\C_\ell$ and
$\C_{\ell,m}$ given by \eqref{eq:DefCl}. It gives
\begin{align*}
    \left[\C_\ell,\C_m\right]
    &=\sum_{k\geq0}t^{k+\ell+m}\sum_{0\leq i\leq k}\frac{\left[\C_{\ell,i},\C_{m,k-i}\right]}{(i+\ell)(k-i+m)}\\
    &=\sum_{k\geq0}\frac{t^{k+\ell+m}}{k+\ell+m}\left(\sum_{0\leq i\leq k}\left[\frac{\C_{\ell,i}}{i+\ell},\C_{m,k-i}\right]+
    \sum_{0\leq i\leq k}\left[\C_{\ell,i},\frac{\C_{m,k-i}}{k-i+m}\right]\right).
\end{align*}
\cref{eq 3 prop com C} implies that it is equal to 0. 
Fix now $m>0$. We have  
\begin{align*}
        \left[\C_0,\C_m\right] &=\sum_{k\geq0}t^{k+m}\sum_{1\leq i\leq k}\frac{\left[\C_{0,i},\C_{m,k-i}\right]}{i(k-i+m)}\\
    &=\sum_{k\geq0}\frac{t^{k+m}}{k+m}\left(\sum_{1\leq i\leq k}\left[\frac{\C_{0,i}}{i},\C_{m,k-i}\right]+
    \sum_{1\leq i\leq k}\left[\C_{0,i},\frac{\C_{m,k-i}}{k-i+m}\right]\right).
    \end{align*}
    Since $\C_{0,0}=1$ and by consequence commutes trivially with any operator, we can write
    \begin{align*}
    \left[\C_0,\C_m\right] 
    &=\sum_{k\geq0}\frac{t^{k+m}}{k+m}\left(\sum_{1\leq i\leq k}\left[\frac{\C_{0,i}}{i},\C_{m,k-i}\right]+
    \sum_{0\leq i\leq k}\left[\C_{0,i},\frac{\C_{m,k-i}}{k-i+m}\right]\right)\\
    &=\sum_{k\geq0}\frac{t^{k+m}}{k+m}\left(m\C_{m+1,k-1}+\C_{m+1,k-1}\right)\\
    &=(m+1)\C_{m+1}.\qedhere
\end{align*}
\end{proof}

\section{The shifted symmetry property}\label{sec shifted sym prop}
The purpose of this section is to prove the following theorem.
\begin{thm}[Shifted symmetry property]\label{thm symmetry}
The function $F^{(k)}\left(t,\bfp,s_1,\dots,s_k\right)$ is $\alpha$-shifted symmetric into the variables $s_1,s_2,\dots,s_k$. 
\end{thm}

We start by proving some general commutation relations which will be useful in the proof of \cref{thm symmetry}.

\subsection{Preliminaries}
\label{subsec:PrelLie}


Let $X$ be a vector space and $\mathcal{O}(X)$ denote the set of linear operators on
$X$. Whenever $A \in \mathcal{O}(X)\llbracket z \rrbracket_+$,
then $\exp(A) := 1+\sum_{n\geq 1}\frac{A^n}{n!}$ is a well-defined element of
$\mathcal{O}(X)\llbracket z \rrbracket$. Note that $\mathcal{O}(X) \llbracket z \rrbracket$ is
a Lie algebra with the standard commutator. For $C \in
\mathcal{O}(X)\llbracket z \rrbracket$ we consider the adjoint action of
$C$, denoted by $\ad_C$, as the linear map defined on the space
$\mathcal{O}(X)\llbracket z \rrbracket$ by 
$$\ad_C(A):=[C,A].$$
By a direct induction we obtain the following lemma.
\begin{lem}\label{lem com ad}
Let $A_1$, $A_2$ and $C$ be three operators such that 
$$\left[A_2,\ad^m_{C}(A_1)\right]=0, \text{ for every }m\geq0.$$
Then 
\begin{equation*}
     \ad_{A_2+C}^m(A_1)=\ad_{C}^m(A_1), \text{ for } m\geq0.
 \end{equation*}
 \end{lem}

We have the following classical
identities between the elements of the Lie algebra $\mathcal{O}(X)\llbracket z
\rrbracket$ (see e.g \cite{Muger2019} for a more general contexts):
\begin{align}\label{eq adjoint derivation}
e^C Ae^{-C}&=e^{\ad_C} A \text{ for } C \in \mathcal{O}(X)\llbracket z \rrbracket_+,\\
\label{eq exp derivation}
  \frac{d}{dt}e^{tA+C}&=e^{tA+C}\frac{e^{\ad_{-C-tA}}-1}{\ad_{-C-tA}}(A)
                        \text{ for } A\in \mathcal{O}(X)\llbracket z \rrbracket, C\in \mathcal{O}(X)\llbracket z \rrbracket_+,
\end{align}
where the last equality holds in $\mathcal{O}(X)\llbracket t,z
\rrbracket$ and
$$\frac{e^{\ad_{-C-tA}}-1}{\ad_{-C-tA}}:=1+\sum_{k\geq 0}\frac{\ad_{-C-tA}^{k}}{(k+1)!}.$$

The following lemma will be quite useful for proving \cref{thm
  symmetry}, and might be interpreted as a special case of the
Baker--Campbell--Hausdorff formula.
 
\begin{lem}\label{lem exp formula}
Let $A,C \in \mathcal{O}(X)\llbracket z \rrbracket_+$ be two operators such that $[A,\ad^m_C(A)]=0$ for every $m\geq0$. 
Then 
\begin{equation*}
  e^{A+C}e^{-C}=\exp\left(\frac{e^{\ad_C}-1}{\ad_C}(A)\right).  
\end{equation*}

\begin{proof}

We consider the following function 
$$\Phi(t):=e^{C}e^{-tA-C}\exp\left(\frac{e^{\ad_{C}}-1}{\ad_C}(tA)\right).$$
We want to prove that $\Phi(1)=1$. Since $\Phi(0)=1$, it is enough then to prove that $\frac{d}{dt}\Phi(t)=0.$
But 
$$\frac{d}{dt}\Phi(t)=e^{C}\left(\frac{d}{dt}e^{-tA-C}\right)\exp\left(\frac{e^{\ad_{C}}-1}{\ad_C}(tA)\right)+e^{C}e^{-tA-C}\frac{d}{dt}\exp\left(\frac{e^{\ad_{C}}-1}{\ad_C}(tA)\right).$$
On one hand, using \cref{eq exp derivation,lem com ad} we have that  
$$\frac{d}{dt}e^{-tA-C}=-e^{-tA-C}\cdot\frac{e^{\ad_{C}}-1}{\ad_C}(A).$$
On the other hand, 
$$\frac{d}{dt}\exp\left(\frac{e^{\ad_{C}}-1}{\ad_C}(tA)\right)=\frac{e^{\ad_{C}}-1}{\ad_C}(A)\exp\left(\frac{e^{\ad_{C}}-1}{\ad_C}(tA)\right).$$

This finishes the proof of the lemma.
\end{proof}
\end{lem}


We deduce the main result of this section.
\begin{prop}\label{prop exponential formula}
Let $A_1,A_2,C \in \mathcal{O}(X)\llbracket z \rrbracket_+$ be three operators such that 
$$\left[\ad_{C}^\ell A_i,\ad_{C}^m A_j\right]=0, \text{ for } 1\leq i,j\leq 2 \text{ and } m,\ell\geq 0.$$

Then 
$$e^{A_1+C}e^{-C}e^{A_2+C}=e^{A_1+A_2+C}.$$

\begin{proof}
We prove that 
$$e^{A_1+C}e^{-C}=\exp\left(\frac{e^{\ad_C}-1}{\ad_C} (A_1)\right)=e^{A_1+A_2+C}e^{-A_2-C}.$$
The left equality in the last line is guaranteed by \cref{lem exp formula}. Let us prove the right equality.
Since the operators $A_1, A_2$ and $C$ satisfy the conditions of \cref{lem com ad}, then 
\begin{equation}\label{eq prop exp formula}
    \ad_{A_2+C}^m(A_1)=\ad_{C}^m (A_1) \text{ for }m\geq0.
\end{equation}
Hence, 
$$[A_1,\ad_{A_2+C}^m(A_1)]=[A_1,\ad_{C}^m (A_1)]=0.$$
By consequence, the operators $A_1$ and $A_2+C$ fulfill then the conditions of \cref{lem exp formula}, and we obtain that 
$$e^{A_1+A_2+C}e^{-A_2-C}=\exp\left(\frac{e^{\ad_{A_2+C}}-1}{\ad_{A_2+C}}(A_1)\right)
= \exp\left(\frac{e^{\ad_{C}}-1}{\ad_{C}}(A_1)\right),$$
where the last equality follows from \cref{eq prop exp formula}. This
finishes the proof.
\end{proof}
\end{prop}

\subsection{Proof of \texorpdfstring{\cref{thm symmetry}}{}}
It will be convenient in this section to separate the constant part of the operator $\Binf$ in the variable $u$. Namely, we consider 
\begin{equation}
  \label{eq:Binf>Def}
  \Binfs(t,\bfp,u):=\sum_{\ell\geq1}u^\ell \C_{\ell}(t,\bfp).
  \end{equation}
\noindent In the following, when there is no ambiguity, we will simply denote 
$$\Binf(u)\equiv \Binf(t,\bfp,u) \text{ and } \Binfs(u)\equiv \Binfs(t,\bfp,u).$$ 

As a consequence of \cref{thm com C}, we have 
\begin{equation}
\left[\Binfs(u),\Binfs(v)\right]=0.
\end{equation}

From \cref{prop expr F} we know that 
$$F^{(k)}(-t,\bfp,
    s_1, \dots,s_k)=\exp\big(\Binf(-\alpha s_1)\big)\cdots \exp\big(\Binf(-\alpha s_k)\big)\cdot1.$$
The purpose of this section is to prove that this function is
symmetric in the shifted variables $(s_i-i/\alpha)_{1 \leq i \leq k}$ and to give a symmetric expression of it. 
\begin{thm}\label{thm symmetry 2}
For every $k\geq1$, we have 
\begin{multline}
F^{(k)}\left(-t,\bfp,s_1,s_2\dots,s_k\right)\\
=\exp\big(\Binf(k-1)\big)\cdots\exp\big(\Binf(1)\big)\exp\big(\C_0+\Binfs(-\alpha
s'_1-k)+\dots+\Binfs(-\alpha s'_k-k)\big)\cdot 1,
\end{multline}
where $s'_i:=s_i-i/\alpha$.
\end{thm}

The following lemma establishes a relation between the operators $\Binfs(u+1)$ and $\Binfs(u)$.

\begin{prop}\label{prop u+1}

For any variable $u$, we have 
$$\Binf(u+1)=\Binf(1)+e^{\C_0}\Binfs(u)e^{-\C_0}.$$
\begin{proof}
Note that $\Binfs(u),\Binf(u),\C_0 \in \mathcal{O}(\mcP[u])\llbracket
t \rrbracket_+$. In particular $\Binf(1) \in \mathcal{O}(\mcP[u])\llbracket
t \rrbracket_+$ is well-defined and we can use formulas from \cref{subsec:PrelLie}. From the definition, we have
\begin{align*}
\Binf(u+1)
&=\sum_{\ell\geq0}(u+1)^\ell \C_\ell \\
&=\sum_{\ell\geq0}\sum_{0\leq k\leq \ell}\binom{\ell}{k}u^k \C_\ell\\
&=\Binf(1)+\sum_{k\geq1}\sum_{\ell\geq k}\binom{\ell}{k} u^k  \C_\ell.
\end{align*}

Applying \cref{thm com C} inductively, we get that  

$$\frac{\ell!}{k!}\C_\ell=\ad_{\C_0}^{\ell-k}\C_k, \text{ for }1\leq k\leq \ell.$$

Hence,
\begin{align*}
\Binf(u+1)
&=\Binf(1)+\sum_{k\geq 1}u^k \left(\sum_{i\geq0 }\frac{\left(\ad_{\C_0}\right)^{i}}{i!}\right) \C_k\\
&=\Binf(1)+\sum_{k \geq 1}u^k e^{\C_0} \C_k e^{-\C_0}\\
&=\Binf(1)+e^{\C_0}\Binfs(u)e^{-\C_0},
\end{align*}
where the second equality follows from \cref{eq adjoint derivation}.
\end{proof}
\end{prop}

We now prove that the operators $\Binfs(u)$, $\Binfs(v)$ and $\C_0$ satisfy the conditions of \cref{prop exponential formula}.
\begin{lem}\label{lem com ad B}
Let $u$ and $v$ be two variables, and let $m$ and $\ell$  be two non negative integers. 
Then 
$$\left[\ad_{\C_0}^\ell \Binfs(u) ,\ad_{\C_0}^m \Binfs(v)\right]=0.$$
\begin{proof}
\cref{eq:Binf>Def} implies that it is enough to prove that
$\left[\ad_{\C_0}^\ell \C_i ,\ad_{\C_0}^m \C_j\right]=0$ for all $i,j
\geq 1$. The latter follows from \cref{thm com C}.
\end{proof}
\end{lem}

\begin{lem}\label{lem u+1}
For any variable $u$, we have
$$\exp(\Binf(u+1))=\exp\big{(}\Binf(1)\big{)}\cdot \exp\big{(}\Binf(u)\big{)}\cdot\exp(-\C_0).$$
\begin{proof}

We know from \cref{prop u+1} that
\begin{align*}
\exp(\Binf(u+1))
&=\exp\big{(}\Binf(1)+e^{\C_0}\Binfs(u)e^{-\C_0}\big{)}\\
&=\exp\big{(}\C_0+\Binfs(1)+e^{\C_0}\Binfs(u)e^{-\C_0}\big{)}.
\end{align*}
But from \cref{lem com ad B} we know that the triplet of operators $\Binfs(1)$, $\Binfs(u)$ and $\C_0$ satisfy the conditions of  \cref{prop exponential formula}. Moreover,  this is also the case for the triplet $\Binfs(1)$, $e^{\C_0}\Binfs(u)e^{-\C_0}$ and $\C_0$, since $e^{\C_0}\Binfs(u)e^{-\C_0}=e^{\ad_{\C_0}}\Binfs(u)$. Hence, 
\begin{align*}
\exp(\Binf(u+1))
&=\exp\big{(}\C_0+\Binfs(1)\big{)}\cdot\exp(-\C_0)\cdot\exp(\C_0+e^{\C_0}\Binfs(u)e^{-\C_0}\big{)}\\
&=\exp\big{(}\C_0+\Binfs(1)\big{)}\cdot\exp(-\C_0)\cdot\exp\left(e^{\C_0}\left(\C_0+\Binfs(u)\right)e^{-\C_0}\right)\\
&=\exp\big{(}\C_0+\Binfs(1)\big{)}\cdot\exp(-\C_0)\cdot\exp\left(e^{\C_0}\Binf(u)e^{-\C_0}\right).
\end{align*}
Finally, by expanding the last exponential, and by observing that for each $\ell\geq0$
$$\exp(-\C_0)\cdot\left(e^{\C_0}\Binf(u)e^{-\C_0}\right)^\ell=\Binf(u)^\ell\cdot\exp(-\C_0),$$
we deduce that
$$\exp(\Binf(u+1))=\exp\big{(}\Binf(1)\big{)}\cdot \exp\big{(}\Binf(u)\big{)}\cdot\exp(-\C_0).\qedhere$$
\end{proof}
\end{lem}

\begin{rmq}\label{rmq B u v}
Fix now two variables $u$ and $v$. By applying \cref{lem u+1} and  \cref{prop exponential formula} with $\Binfs(u)$, $\Binfs(v)$ and $\C_0$, we get that 
\begin{align}\label{eq sym B}
\exp(\Binf(u+1))
&\exp(\Binf(v))\\
&=\exp\big{(}\Binf(1)\big{)}\cdot\exp\big{(}\C_0+\Binfs(u)\big{)}\cdot\exp(-\C_0)\cdot\exp(\Binf(v))\nonumber\\
&=\exp\big{(}\Binf(1)\big{)}\cdot\exp\big{(}\C_0+\Binfs(u)+\Binfs(v)\big{)}\nonumber.
\end{align}
In particular, we deduce that
\begin{equation}\label{eq sym}
    \exp\big(\Binf(u+1)\big)\exp\big(\Binf(v)\big)=\exp\big(\Binf(v+1)\big)\exp\big(\Binf(u)\big).
\end{equation}
One can see that this together with \cref{prop expr F} imply that the series $F^{(2)}\left(-t,\bfp,
    s_1,s_2\right)$ is symmetric in the variables $s_1-1/\alpha,s_2-2/\alpha$.
\end{rmq}

We now prove the main theorem of this section.
\begin{proof}[Proof of \cref{thm symmetry 2}]
We prove it by induction on $k$. For $k=1$ the result is straightforward from the definition and for $k=2$ it corresponds to \cref{rmq B u v}.
Fix $k\geq0$ and suppose that the theorem holds for
$F^{(k)}$. \cref{prop expr F} implies that
\begin{multline*}
F^{(k+1)}\left(-t,\bfp,
    s_1 , s_2 ,\dots,s_{k+1} 
\right)\\
=\exp\big(\Binf(-\alpha s'_{1}-1)\big)\cdot\exp\big(\Binf(k-1)\big)\cdots\exp\big(\Binf(1)\big)\cdot\\
\quad\exp\Big(\C_0+\Binfs(-\alpha s'_2-k-1)+\cdots+\Binfs(-\alpha s'_{k+1}-k-1)\Big)\cdot1.
\end{multline*}

Using $k-1$ times \cref{eq sym}, we obtain
\begin{multline*}
F^{(k+1)}\left(-t,\bfp,
    s_1, s_2,\dots,s_{k+1} 
\right)\\
=\exp\left(\Binf(k)\right)\cdots\exp\left(\Binf(2)\right)\cdot\exp\big(\Binf(-\alpha s'_{1}-k)\big)\cdot\\
\exp\Big(\C_0+\Binfs(-\alpha s_2-k-1)+\cdots+\Binfs(-\alpha s'_{k+1}-k-1)\Big)\cdot1.
\end{multline*}

Using \cref{lem u+1}, this can be rewritten as follows
\begin{multline*}
F^{(k+1)}\left(-t,\bfp,
    s_1, s_2,\dots,s_{k+1} 
\right)\\
=\exp\left(\Binf(k)\right)\cdots\exp\left(\Binf(1)\right)\cdot\exp(\Binf(-\alpha s'_{1}-k-1))\cdot\exp(-\C_0)\cdot\\
\exp\Big(\C_0+\Binfs(-\alpha s'_2-k-1)+\dots+\Binfs(-\alpha s'_{k+1}-k-1)\Big)\cdot1.
\end{multline*}
Finally, \cref{lem u+1} allows us to apply \cref{prop exponential formula} with the operators $\Binfs(-\alpha s'_1-k-1)$, $\Binfs(-\alpha s'_2-k-1)+\dots+\Binfs(-\alpha s'_{k+1}-k-1)$ and $\C_0$ in order to reassemble the last three exponentials, which concludes the proof of the theorem.
\end{proof}

\section{Proofs of the main results}\label{sec proof first main result}

\subsection{End of proof of \cref{thm Jack char}}
\begin{lem}\label{lem proj lim}
The functions $F^{(k)}$ satisfy the condition of \cref{eq shifted functions}:
\begin{align*}
    F^{(k+1)}(t,\bfp,s_1,\dots,s_k,0)=F^{(k)}(t,\bfp,s_1,\dots,s_k).
\end{align*}
As a consequence, the projective limit $F^{(\infty)}:=\varprojlim F^{(k)}$ is well defined in $\mathcal S^*_\alpha\llbracket t,p_1,p_2,\dots\rrbracket$, and 
$$F^{(\infty)}(t,\bfp,s_1,s_2,\dots)=\sum_{M\in \Minf}(-t)^{|M|}p_{\tf(M)}\frac{b^{\vartheta_\rho(M)}}{2^{|\Vbul(M)|-\cc(M)}\alpha^{cc(M)}}\prod_{i\geq 1}\frac{(-\alpha s_{i})^{|\mathcal{V}_\circ^{(i)}(M)|}}{z_{\nu_\bullet^{(i)}(M)}}.$$
\begin{proof}
We know from the recursive expression of $F^{(k)}$ given in \cref{prop expr F} that 
\begin{equation*}
F^{(k+1)}\left(t,\bfp,s_1,\dots,s_{k},0\right)=\exp\left(\Binf(-t,\bfp,-\alpha s_{1})\right)\cdots\exp\left(\Binf(-t,\bfp,-\alpha s_{k})\right) F^{(1)}\left(t,\bfp,0\right).
\end{equation*}
Using the combinatorial expression of the function $F^{(1)}$ (see
\cref{eq def F}), one can see that $F^{(1)}\left(t,\bfp,0\right)=1$,
since the only bipartite map without any white vertex is the empty map
(by convention). 
Reapplying \cref{prop expr F}, the last equation is equal to $F^{(k)}(t,\bfp,s_1,\dots,s_k)$.
The second part of the lemma is a consequence of \cref{thm symmetry}.
\end{proof}
\end{lem}

We now finish the proof of the first main result.
\begin{proof}[Proof of \cref{thm Jack char}]
Fix a partition $\mu$. It is enough to prove that the coefficient of $t^{|\mu|}p_\mu$ in $F^{(\infty)}(t,\bfp,\lambda)$ satisfies the conditions of \cref{thm Feray}.  The fact that this coefficient vanishes on partitions $\lambda$ of size  $|\lambda|<|\mu|$ is given by \cref{thm vanishing}, and we know that it is $\alpha$-shifted symmetric from \cref{lem proj lim}. 

Let us now prove that the top homogeneous part in $[t^{|\mu|}p_\mu]F^{(k)}(t,\bfp,\lambda)$ is equal to $\frac{\alpha^{|\mu|-\ell(\mu)}}{z_\mu}p_\mu(\lambda_1,\lambda_2,\dots,\lambda_k)$ for any $k\geq1$.  
This part corresponds to labelled $k$-layered maps of face-type $\mu$ and with maximal number of white vertices, \textit{i.e} maps with $|\mu|$ white vertices which are all of degree 1.
Note that adding a black vertex connected to $n$ white vertices in a layer $i$ of a map $M$ corresponds to  multiplying its marking $\kappa(M)$ by $p_n (-\alpha \lambda_i)^n/\alpha$. Thus, in order to obtain the top homogeneous part in $F^{(k)}$ we replace $\B_n(\bfp,\lambda_i)$ by $p_n (-\alpha \lambda_i)^n$ in  \cref{eq F}.
As a consequence, the top homogeneous part in $[t^{|\mu|}p_\mu]F^{(k)}$ is given by 
\begin{align*}
  [t^{|\mu|}p_\mu]\exp\left(\sum_{n\geq 1}\frac{(-t)^{n}\cdot (-\alpha \lambda_1)^{n} p_n}{\alpha n}\right)&\cdots \exp\left(\sum_{n\geq 1}\frac{(-t)^{n}\cdot (-\alpha \lambda_k)^n p_n}{\alpha n}\right)\\
  &=[t^{|\mu|}p_\mu]\exp\left(\sum_{n\geq 1}\frac{t^{n} \alpha^{n-1}\cdot p_n(\lambda_1,\dots,\lambda_k) p_n}{ n}\right)\\
  &=\frac{\alpha^{|\mu|-\ell(\mu)}}{z_\mu} p_\mu(\lambda_1,\dots,\lambda_k).\qedhere
\end{align*}
\end{proof}

\subsection{Proof of \cref{thm Jack via diff op}}\label{ssec Jack via diff}
In this section, we give a direct way to construct Jack polynomials by
adding the rows in increasing order of their size. This formula is
more efficient to generate Jack polynomials than the one given by \cref{prop expr F}. 
\begin{thm}\label{thm Jack via diff op 2}
    Fix a partition $\lambda$. Let $\mu$ be the partition obtained from $\lambda$ by removing the largest part; $\mu=\lambda\backslash \lambda_1$. Then 
    \begin{align*}
      J^{(\alpha)}_\lambda
      &=[t^{\lambda_1}]\exp\left(\Binf(-t,\bfp,-\alpha \lambda_1)\right)\cdot J^{(\alpha)}_\mu\\
      &=\sum_{\ell \geq 1}\sum_{n_1,\dots,n_\ell \geq 1\atop
        n_1+\cdots +n_\ell=\lambda_1}\frac{1}{\ell!}\frac{\B_{n_1}(\bfp,-\alpha \lambda_1)}{n_1}\cdots \frac{\B_{n_\ell}(\bfp,-\alpha \lambda_1)}{n_\ell}\cdot J^{(\alpha)}_\mu.
    \end{align*}
\begin{proof}
    Let $|\lambda| = n$, and $\ell(\lambda) = \ell$. From the definition of the Jack characters and \cref{thm Jack char} and \cref{prop expr F}, we have
    \begin{align*}
      J^{(\alpha)}_\lambda
      &=[t^{|\lambda|}]\exp\left(\Binf(-t,\bfp,-\alpha \lambda_1)\right)\cdots\exp\left(\Binf(-t,\bfp,-\alpha \lambda_{\ell(\lambda)})\right)\cdot 1\\
      &=\sum_{n_1,\dots,n_\ell \geq 1\atop
        n_1+\cdots +n_\ell=n} \left([t^{n_1}]\exp\left(\Binf(-t,\bfp,-\alpha \lambda_{1})\right)\right)
\cdots  \left([t^{n_{\ell}}]\exp\left(\Binf(-t,\bfp,-\alpha \lambda_{\ell(\lambda)})\right)\right)\cdot 1.
    \end{align*}
But from the proof of \cref{thm vanishing}, the only choice for
$(n_i)_{1 \leq i \leq \ell}$
that does not give a null term is $n_i = \lambda_i$ for all $1 \leq i
\leq \ell$. Indeed, otherwise there exists $i$ such that $n_i>\lambda_i$, and from \cref{eq 2 thm vanishing}, we have
    $$\bigl([t^{n_{i}}]\exp\left(\Binf(-t,\bfp,-\alpha \lambda_{i})\right)\bigr)\cdots\bigl([t^{n_{\ell}}]\exp\left(\Binf(-t,\bfp,-\alpha \lambda_{\ell})\right)\bigr)\cdot 1=0.$$

    As a consequence, we obtain 
    \begin{equation}
        J^{(\alpha)}_\lambda=\left([t^{\lambda_1}]\exp\left(\Binf(-t,\bfp,-\alpha \lambda_{1})\right)\right)
\cdots  \left([t^{\lambda_{\ell}}]\exp\left(\Binf(-t,\bfp,-\alpha \lambda_{\ell})\right)\right)\cdot 1.
    \end{equation}
    Comparing this expression for the partitions $\lambda$ and $\mu$ we obtain the first part of the equation of the theorem.
    To obtain the second part of the equation we expand the
    exponential and we use the fact that the operators $\B_n(\bfp,-\alpha \lambda_1)$ commute for $n\geq 1$, see \cref{prop comm B}.
\end{proof}
\end{thm}

\subsection{Positivity in Lassalle's conjecture}

We consider two sequences of variables $s_1,\dots,s_k$ and $r_1,\dots,r_k$.

We introduce the generating series 
$\tF^{(k)}\left(t,\bfp,\begin{array}{ccc}
    s_1 &\dots&s_k  \\
    r_1 & \dots& r_k
\end{array}\right)$ 
inductively by $F^{(0)}=1$ and for every $k\geq0$ 
$$\tF^{(k+1)}\left(t,\bfp,\begin{array}{ccc}
    s_1 &\dots&s_{k+1}  \\
    r_1 & \dots& r_{k+1}
\end{array}\right)=\exp\left(r_{k+1}\Binf(-t,\bfp,-\alpha s_{k+1})\right)\cdot \tF^{(k)}\left(t,\bfp,\begin{array}{ccc}
    s_1 &\dots&s_k  \\
    r_1 & \dots& r_k
\end{array}\right).$$

Hence the functions $F^{(k)}$ defined in \cref{eq F} are obtained as a specialization of $\tF^{(k)}$:
$$F^{(k)}(t,\bfp,s_1,\dots ,s_k)=\tF^{(k)}\left(t,\bfp,\begin{array}{ccc}
    s_1 &\dots&s_k  \\
    1 & \dots& 1
    \end{array}\right).
$$

Using the combinatorial interpretation for the operator $\Binf$ from
\cref{ssec cat Y}, the same argument that allowed to interpret the
function $F^{(k)}$, also gives the following interpretation for $\tF^{(k)}$:
\begin{equation}\label{eq tilde F}
    \tF^{(k)}\left(t,\bfp,\begin{array}{ccc}
    s_1 &\dots&s_k  \\
    r_1 & \dots& r_k
    \end{array}\right)=\sum_{M\in \Mk}\frac{(-t)^{|M|}p_{\nu_\diamond(M)}b^{\vartheta_\rho(M)}}{2^{|\Vbul(M)|-\cc(M)}\alpha^{cc(M)}}\prod_{1\leq i\leq k}\frac{r_{i}^{|\Vbul^{(i)}(M)|}(-\alpha s_{i})^{|\mathcal{V}_\circ^{(i)}(M)|}}{z_{\nu_\bullet^{(i)}(M)}}.
\end{equation}

The two following  properties follow from the definition of the functions $\tF^{(k)}$:

\begin{enumerate}[label=(\roman*)]
    \item For every $1\leq i\leq k-1$, 
    \begin{equation*}\label{eq prop tF 1}
      \tF^{(k)}\left(t,\bfp,\begin{array}{ccc}
    s_1 & \dots &s_k \\
    r_1 & \dots & r_k
\end{array}\right)\bigg|_{s_i=s_{i+1}=s}=
\tF^{(k-1)}\left(t,\bfp,\begin{array}{cccccc}
    s_1 & \dots & s& s_{i+2}&\dots& s_k \\
    r_1 & \dots & r_i+r_{i+1 } & r_{i+2}& \dots & r_k
\end{array}\right).  
    \end{equation*}

    \item For every $1\leq i\leq k$,
    \begin{equation*}\label{eq prop tF 2}
      \tF^{(k)}\left(t,\bfp,\begin{array}{ccc}
    s_1 & \dots &s_k \\
    r_1 & \dots & r_k
\end{array}\right) \bigg|_{r_i=0}
=\tF^{(k-1)}\left(t,\bfp,\begin{array}{cccccc}
    s_1 & \dots & s_{i-1}& s_{i+1}&\dots& s_k \\
    r_1 & \dots & r_{i-1}& r_{i+1}& \dots & r_k
\end{array}\right)  
    \end{equation*}
\end{enumerate}

 We deduce the following proposition.

\begin{prop}\label{prop tF}
Let $\lambda$ be a partition, and let $\left(\begin{array}{ccc}
    s_1 & \dots &s_k \\
    r_1 & \dots & r_k
\end{array}\right)$ be multirectangular coordinates of $\lambda$. Then 
$$\tF^{(k)}\left(t,\bfp,\begin{array}{ccc}
    s_1 & \dots &s_k \\
    r_1 & \dots & r_k
\end{array}\right)=F^{(\ell(\lambda))}(t,\bfp,\lambda_1,\dots,\lambda_{\ell(\lambda)}).$$ 
In particular, the quantity $\tF^{(k)}\left(t,\bfp,\begin{array}{ccc}
    s_1 & \dots &s_k \\
    r_1 & \dots & r_k
\end{array}\right)$ does not depend on the multirectangular coordinates of $\lambda$ chosen. 
\begin{proof}
We start by removing all the pairs $(s_i,r_i)$ for which $r_i=0$,
using property $(ii)$ above. Then, we use property $(i)$ in order to
decrease the remaining coordinates $r_i$ until they are all equal
1. More precisely, we apply multiple times the following equation
    \begin{equation*}
      \tF^{(k)}\left(t,\bfp,\begin{array}{ccc}
    s_1 & \dots &s_k \\
    r_1 & \dots & r_k
\end{array}\right)
=\tF^{(k+1)}\left(t,\bfp,\begin{array}{cccccc}
    s_1 & \dots & s_{i}& s_{i}&\dots& s_k \\
    r_1 & \dots & r_{i}-1& 1& \dots & r_k
\end{array}\right).
    \end{equation*}
Note that in each one of these operations the new coordinates obtained are also multirectangular coordinates for $\lambda$. Hence when $r_i=1$ for every $i$, we have $k=\ell(\lambda)$ and $s_i=\lambda_i$ for each $1\leq i\leq\ell(\lambda)$.
\end{proof}
\end{prop}

As a consequence of \cref{thm Jack char} and \cref{prop tF}, we obtain that for any partition~$\mu$
\begin{equation}\label{eq Jch tF}
  \tJch_\mu(\bfs,\bfr)=[t^{|\mu|}p_\mu]\tF^{(k)}\left(t,\bfp,\begin{array}{ccc}
    s_1 & \dots &s_k \\
    r_1 & \dots & r_k
\end{array}\right),  
\end{equation}
where $\bfs=(s_1\dots,s_k,0,\dots)$ and $\bfr=(r_1,\dots,r_k,0,\dots)$.
We now prove the positivity in Lassalle's conjecture.

\begin{proof}[Proof of positivity in \cref{thm Lassalle conj}]
Comparing \cref{eq tilde F} and \cref{eq Jch tF} and taking the limit on $k$ we get
\begin{equation}\label{main eq}
  \tJch_\mu(\bfs,\bfr)=\sum_{M}\frac{b^{\vartheta_\rho(M)}}{2^{|\Vbul(M)|-\cc(M)}\alpha^{cc(M)}}\prod_{i\geq 1}\frac{r_{i}^{|\Vbul^{(i)}(M)|}(-\alpha s_{i})^{|\mathcal{V}_\circ^{(i)}(M)|}}{z_{\nu_\bullet^{(i)}(M)}},  
\end{equation}
where the sum is taken over layered maps $M$ of face type $\mu$. Since every connected component of a bipartite map contains at least one white vertex, the $\alpha$-term which appears in the denominator is compensated. This concludes the proof of the theorem.
\end{proof}

\section*{Acknowledgement}

We extend our gratitude to Guillaume Chapuy, Valentin F\'eray, and Piotr Śniady for their invaluable collaboration and discussions over many years regarding related problems. Without their contributions, this research would not have been possible.

\bibliographystyle{amsalpha}
\bibliography{biblio2015.bib}

\end{document}